\documentclass[a4paper,12pt]{article}    
\usepackage{amsfonts}
\usepackage{amsthm}
\usepackage{amscd}
\usepackage{amsmath,amssymb}
\usepackage[utf8]{inputenc}
\usepackage{enumerate}
\usepackage{graphicx}

\newtheorem{thm}{Theorem}[section]
\newtheorem{cor}[thm]{Corollary}
\newtheorem{lem}[thm]{Lemma}
\newtheorem{prop}[thm]{Proposition}
\newtheorem{claim}[thm]{Claim}

\theoremstyle{definition} 
\newtheorem{defn}[thm]{Definition}
\theoremstyle{definition}
\newtheorem{rem}[thm]{Remark}
\newtheorem{exe}[thm]{Example}
\newtheorem{nota}[thm]{Notation}

\numberwithin{equation}{section}

\newcommand{\R}{\mathbb{R}}

\newcommand{\Act}{\mathbb{L}}

\newcommand\INT[1]{\mathaccent'27{#1}}

\newcommand\Graph{\operatorname{Graph}}
\title{Singularities of solutions of time dependent Hamilton-Jacobi equations. Applications
to Riemannian geometry}
\author{Piermarco Cannarsa\footnote{Dipartimento di Matematica\/\/, Universit\`a di Roma ``Tor Vergata'', Via della Ricerca Scientifica 1, 00133 Roma, Italy. cannarsa@mat.uniroma2.it}, Wei Cheng\footnote{Department of Mathematics, Nanjing University, Nanjing 210093, China. chengwei@nju.edu.cn}\ \&\ Albert Fathi\footnote{Georgia Institute of technology \& ENS de Lyon (Emeritus), School of Mathematics, Atlanta, GA 30332, USA.
albert.fathi@math.gatech.edu}}
\date{December 2019}

\begin{document}
\maketitle

\begin{abstract}
 
If $U:[0,+\infty[\times M$ is a uniformly continuous viscosity solution of 
the evolution Hamilton-Jacobi equation
$$\partial_tU+ H(x,\partial_xU)=0,$$
where $M$ is a not necessarily compact manifold, and $H$ is a Tonelli Hamiltonian,
we prove the set $\Sigma(U)$, of points where $U$ is  not differentiable, is locally contractible. Moreover,
we study the homotopy type of $\Sigma(U)$. We also give an application to the singularities of a distance function
to a closed subset of a complete Riemannian manifold.
\end{abstract}

\section{Introduction}
Let $M$ be a   smooth connected but not necessarily compact manifold. 
We will assume $M$ endowed with a {\it complete} Riemannian metric $g$.
For $v\in T_xM$, the norm $\Vert v\Vert_x$ is $g(v,v)^{1/2}$.
We will denote also by $\Vert\cdot\Vert_x$ the dual norm on $T^*_xM$.
If $\gamma:[a,b]\to M$ is a curve, its length $\ell_g(\gamma)$ (for the metric g) 
is defined by
$$\ell_g(\gamma)=\int_a^b\lVert \dot\gamma(s)\rVert_{\gamma(s)}\,ds.$$
The distance $d$ that we will use on $M$ is the Riemannian distance obtained 
from the Riemannian metric. Namely, if $x,y\in M$ the distance $d(x,y)$ is the infimum of the length
of curves joining $x$ to $y$.
Since $g$ is complete, the distance $d$ is complete, and for every pair of points $x,y\in M$, there
exists a curve joining $x$ to $y$ whose length is $d(x,y)$.
Moreover, for every compact subset $K\subset M$ and every finite $R\geq 0$, the closed 
$R$-neighborhood $\bar V_R(K)$ of $K$ defined by
$$\bar V_R(K)=\{x\in M\mid d(x,K)\leq R\}$$
is itself compact. 

Before giving our results for the general Hamilton-Jacobi equation, we will give the consequences in 
Riemannian geometry. 

If $C$ is a closed subset of the complete Riemannian manifold $(M,g)$. As usual, the distance function 
$d_C:M\to[0,+\infty[$ to $C$  is defined by
$$d_C(x)=\inf_{c\in C}d(c,x).$$
We will denote by $\Sigma^*(d_C)$ the set of points in $M\setminus C$ where $d_C$ is not differentiable.
Note that $\Sigma^*(d_C)$ has Lebesgue measure $0$, since the Lipschitz function $d_C$ on $M$ is differentiable
almost everywhere.
\begin{thm}\label{LocContrDitFunc} Consider the closed subset $C$ of the complete Riemannian manifold $(M,g)$. 
Then  $\Sigma^*(d_C)$ is locally contractible.
\end{thm}
As a first application, if we take $C=\{p\}$, the set $\Sigma^*(d_p)$ is nothing but the set of $q\in M$ such that
there exists two distinct minimizing geodesics from $p$ to $q$. The closure is known as the the cut locus 
$\operatorname{Cut-locus}_{(M,g)}(p)$ of $p$ for $(M,g)$. It is well-known that, for $M$ compact,
this cut locus $\operatorname{Cut-locus}_{(M,g)}(p)$ is a deformation retract of $M\setminus p$, therefore it is
locally contractible. However, even if there is an extensive literature on the cut locus, very little was known up to now
about the set of $q\in M$ such that
there exists two distinct minimizing geodesics from $p$ to $q$. 
As Marcel Berger said in  \cite[Page 284]{Berger}:

{\sl The difficulty for all these studies is an unavoidable dichotomy for cut points: the mixture of
points with two different segments and conjugate points.}

Our methods permit to separate the study of these two sets.
 
To state another consequence of Theorem \ref{LocContrDitFunc},
we introduce the following definition: 

\begin{defn}\label{DefUNU} If $(M,g)$ is a complete Riemannian manifold, we define
 the subset ${\cal U}(M,g)\subset M\times M$ as the set of $(x,y)\in M\times M$ such 
that there exists a \emph{unique} minimizing $g$-geodesic between $x$  and $y$. This set $ {\cal U}(M,g)$
contains a neighborhood of the diagonal $\Delta_M\subset M\times M$. The complement
${\cal NU}(M,g)=M\times M\setminus {\cal U}(M,g)$ is the set of points $(x,y)\in M$ such that
there exists at least two distinct minimizing  $g$-geodesics between $x$  and $y$. 
\end{defn} 
In fact, as we will see in Example \ref{Exe5}, we have ${\cal NU}(M,g)=\Sigma^*(d_{\Delta M})$,
the set of singularities in $M\times M\setminus \Delta_M$ of the distance function of points in $M\times M$
to the closed subset $\Delta_M$. Therefore, Theorem \ref{LocContrDitFunc} implies: 
\begin{thm}\label{LCNU} For every complete Riemannian manifold $(M,g)$, the set
${\cal NU}(M,g)\subset M\times M\setminus  \Delta_M$ is locally contractible.
In particular, the set ${\cal NU}(M,g)$ is locally path connected.
\end{thm}
As above, from ${\cal NU}(M,g)=\Sigma^*(d_{\Delta M})$, we obtain that ${\cal NU}(M,g)$ has Lebesgue measure $0$.

\begin{defn}\label{AubryClosed}
	For a closed subset $C\subset M$, we define its Aubry set ${\cal A}^*(C)$ 
	as the set of points $x\in M\setminus C$ such that
	there exists a curve $\gamma:[0,+\infty[\to M$ parameterized by arc-length such that $d_C(\gamma(t))=t$ and 
	$x=\gamma(t_0)$ for some $t_0>0$.
\end{defn}
\begin{rem}\label{RemdCcal}
	A curve $\gamma:[0,+\infty[\to M$ parameterized by arc-length such that $d_C(\gamma(t))=t$ is necessary a $g$-minimizing geodesic. In fact, for $0\leq s\leq t$, we have 
	\begin{align*}
	d_C(\gamma(t))\leq d_C(\gamma(s)) + d(\gamma(s),\gamma(t))\\
    &\leq s +\ell_g(\gamma|[s,t])\\
	&\leq s+(t-s)=t.
	\end{align*}
	Since $d_C(\gamma(t))=t$, we obtain $d(\gamma(s),\gamma(t))=\ell_g(\gamma|[s,t])=t-s$, which means that the curve $\gamma$, parametrized by arc-length, is a
	minimizing $g$-geodesic.
\end{rem}
We necessarily have $\Sigma^*(d_C)\cap {\cal A}^*(C)=\emptyset$,
see below.
\begin{thm}\label{GlobalDistFunc} If $C$ is a closed subset of the complete Riemannian manifold $(M,g)$,
then the inclusion $\Sigma^*(d_C)\subset M\setminus(C\cup {\cal A}^*(C))$ is a homotopy equivalence.
\end{thm}
If $U$ is a bounded connected component of $M\setminus C$, then $U\cap {\cal A}^*(C)=\emptyset$,
see \S \ref{SecDist}, and Theorem \ref{GlobalDistFunc} implies that the inclusion $\Sigma^*(d_C)\cap  U\subset
U$ is a homotopy equivalence. This fact was already known. It is due to Lieutier \cite{Lieu} in the Euclidean case
and to Albano, Cannarsa, Nguyen \& Sinestrari \cite{ACNS} in the general Riemannian case. The non-compact
case is, to our knowledge, new, see however \cite{CanPei} where they study the unbounded components of
$\Sigma^*(d_C)$ in the Euclidean case.

A consequence of Theorem \ref{GlobalDistFunc} is
\begin{thm}\label{HomTypeNU} For every compact connected Riemannian manifold $M$, the inclusion
${\cal NU}(M,g)\subset M\times M\setminus  \Delta_M$ is a homotopy equivalence.
Therefore the set ${\cal NU}(M,g)$ is path connected and even locally contractible.
\end{thm}
Of course, the homotopy equivalence in this theorem is a consequence of the compact version of
Theorem \ref{GlobalDistFunc}, which is due,
as we said above, to Albano, Cannarsa, Nguyen \& Sinestrari \cite{ACNS}.

We will give a version for non-compact $M$ in \S \ref{SecDist}.

For sake of completeness, we note that the subset $ \Delta_M$ is a deformation retract of ${\cal U}(M,g)$. In fact, we can get 
such a retraction using the midpoint in a geodesic segment minimizing the length between the pair of points.

We now state our general results for Tonelli Hamiltonians. The local contractibility Theorem \ref{LocContr}  is valid under
slightly less restrictive conditions on the Hamiltonian $H$, see \S \ref{SecLocCon}.

We recall that a Tonelli Hamiltonian $H: T^*\/M\to \R$ on $M$ 
(for the complete Riemannian metric $g$) is a function $H: T^*\/M\to \R$ that satisfies the following conditions:
\begin{enumerate}[(1*)]
\item  The Hamiltonian $H$ is at least $C^2$
\item\label{HUnSu} (Uniform superlinearity) For every $K\geq 0$, we have
$$C^*(K)=\sup_{(x,p)\in T^*\/M} K\lVert p\rVert_x-H(x,p)<\infty.$$
\item (Uniform boundedness in the fibers) For every $R\geq 0$, we have
$$A^*(R)=\sup \{H(x,p)\mid \Vert p\Vert_x\leq R\}<+\infty.$$
\item ($C^2$ strict convexity in the fibers) For every $(x,p)\in T^*\/M$, 
the second derivative along the fibers $\partial^2 H /\partial 
p^2(x,p)$ is positive definite.
\end{enumerate}
Note that (\ref{HUnSu}*) implies
$$\forall (x,p)\in T^*\/M, H(x,p)\geq K\Vert p\Vert_x -C^*(K).$$

We will consider viscosity solutions of the Hamilton-Jacobi equation. There are several
classical introductions to viscosity solutions \cite{Barles, BarCap, CanSin, Evans}. 
The more recent introductions \cite{FatEcos, FatNC} are well-adapted to our manifold setting.

If $u:N\to \R$ is a function defined on the manifold $N$, a singularity of $u$ is a point of $N$ where  
$u$ is not differentiable. We denote by $\Sigma(u)$ the set of singularities of $u$.

The goal of this work is to study the topological structure of the set of singularities $\Sigma(U)$, with 
$U:O\to\R$  a continuous viscosity solution of the evolutionary Hamilton-Jacobi equation
\begin{equation}\label{HJE0}
\partial_tU+ H(x,\partial_xU)=0,
\end{equation}
defined on the open subset $O\subset \R\times M$.

In \cite{CCF}, we announced the results and sketched the proofs for $M$ compact in the case of the stationary 
Hamilton-Jacobi equation, i.e. for $U$ of the form $U(t,x)=u(x)-ct$, with $u:M\to \R$ and $c\in\R$.
We extend the results of \cite{CCF} to the case of the evolutionary Hamilton-Jacobi equation  \eqref{HJE0} covering also the case when $M$ is non-compact.

Our first result is a local contractibility result.
\begin{thm}\label{LocContr} Let $H:T^*\/M\to \R$ be a Tonelli Hamiltonian. If the   function $U:O\to \R$, defined on the open subset $O\subset \R\times M$
is a continuous viscosity solution, on $]0,t[\times M$,  of the evolutionary Hamilton-Jacobi equation
\begin{equation}
\partial_tU+ H(x,\partial_xU)=0,
\end{equation}
then set $\Sigma(U)\subset O$ of singularities of $U$ is locally contractible.
\end{thm}
In fact, as we will see, the above theorem follows from its particular case with $U:]0,+\infty[\times M\to \R$.

To give  a more global result on the topology of $\Sigma(U)$ we need  the Aubry set of a solution of \eqref{HJE0}.
For this we first recall that  
the Lagrangian $L:TM\to \R$ (associated to $H$) is defined
by 
$$L(x,v)=\sup_{p\in T^*_xM}p(v)-H(x,p).$$
This Lagrangian $L$ is finite everywhere, and enjoys the same properties as $H$, namely
\begin{enumerate}[(1)]
\item\label{TonLag1} The Lagrangian $L$ is at least $C^2$ (in fact, it is as smooth as $H$).
\item (Uniform superlinearity) For every $K\geq 0$, we have
\begin{equation}\label{TonLag2} C(K)=\sup_{(x,v)\in TM} K\lVert v\rVert_x-L(x,v)<\infty.
\end{equation}
\item(Uniform boundedness in the fibers) For every $R\geq 0$, we have
\begin{equation}\label{TonLag3} 
A(R)=\sup \{L(x,v)\mid \Vert v\Vert_x\leq R\}<+\infty.
\end{equation}
\item\label{TonLag4} ($C^2$ strict convexity in the fibers) for every $(x,v)\in TM$, 
the second derivative along the fibers $\partial^2 L /\partial 
v^2(x,v)$ is positive definite.
\end{enumerate}
Again \eqref{TonLag2} implies
\begin{equation}\label{TonLag2bis} 
\forall (x,v)\in TM, L(x,v)\geq K\Vert v\Vert_x -C(K).
\end{equation}

A Lagrangian $L:TM\to \R$, on the complete Riemannian manifold $(M,g)$, is said to be Tonelli if it satisfies the
conditions \eqref{TonLag1} to \eqref{TonLag4} above.

\begin{defn}[{\bf Aubry set}]\label{AubryEvol} Let $U:]0,T[\times M\to \R$, with $T\in]0,+\infty]$, be a viscosity solution, on $]0,T[ \times M$, of
 the evolutionary Hamilton-Jacobi equation \eqref{HJE0}. The Aubry set ${\cal I}_T(U)$ of $U$ is the set of points 
 $(t,x)\in ]0,T[\times M$ for which we can find a curve $\gamma:[0,T[\to M$, with 
 $\gamma(t)=x$ and
 $$U(b, \gamma(b))-U(a, \gamma(a))=\int_a^bL(\gamma(s),\dot\gamma(s))\,ds,$$
 for every $a<b\in [0,T[$.
\end{defn}
It is well-known that $U$ is differentiable at every point of ${\cal I}_T(U)$, see Proposition \ref{DiffProp}.
Therefore, we have $\Sigma(U)\cap  {\cal I}_T(U)=\emptyset$.
To avoid further machinery, in this introduction, we will state our results assuming the function $U:[0,t]\times M\to\R$
uniformly continuous.

\begin{thm}\label{HomGlob} Let $H:T^*\/M\to \R$ be a Tonelli Hamiltonian. 
Assume that the uniformly continuous function $U:[0,t]\times M\to \R$ 
is a viscosity solution, on $]0,t[\times M$,  of the evolutionary Hamilton-Jacobi equation \eqref{HJE0}. Then the inclusion
$\Sigma_t(U)=\Sigma(U)\cap ]0,t[ \times M\subset ]0,t[ \times M\setminus {\cal I}_t(U)$ 
is a homotopy equivalence.
\end{thm}
\noindent {\bf Acknowledgments.} This work is partly supported by National Natural Scientific Foundation of China (Grant No.11790272, No.11871267 and No.11631006), and the National Group for Mathematical Analysis, Probability and Applications (GNAMPA) of the Italian Istituto Nazionale di Alta Matematica ``Francesco Severi''. The authors acknowledge support from the MIUR Excellence Department Project awarded to the Department of Mathematics, University of Rome Tor Vergata, CUP E83C18000100006. The authors are also grateful to the departments of mathematics of Fudan University and Nanjing University, for the hospitality they received from these institutions while completing this manuscript.

Albert Fathi's work was also supported by ANR-12-BS01-0020 WKBHJ, and  the National Science Foundation 
under Grant No. DMS-1440140 while he was in residence as a Clay Senior Scholar at the Mathematical 
Sciences Research Institute in Berkeley, California, during the Fall 2018 semester.

\section{Background}
We will need to use some of the facts about viscosity solutions and the  negative Lax-Oleinik semi-groups. 
We refer to \cite{FatEcos} and \cite{FatNC} for details and proofs.

In the remainder of this section, we will assume that $H:T^*\/M\to\R$ is a given Tonelli Hamiltonian on 
the complete Riemannian manifold $M$. We will denote by $L:TM\to\R$ its associated 
Lagrangian defined by
$$L(x,v)=\sup_{p\in T^*_xM}p(v)-H(x,p).$$
\subsection{Action and minimizers}
If $\gamma: [a,b]\to M$ is an absolutely continuous curve, its action $\Act(\gamma)$ is defined by
 $$\Act(\gamma)=\int_a^bL(\gamma(s),\dot\gamma(s))\, ds.$$
Note that since $L$ is bounded from below (by $-C(0)$), we always have $\Act(\gamma)>-\infty$, although we may have $\Act(\gamma)=+\infty$.

For $x,y\in M$, and $t>0$, the minimal action $h_t(x,y)$ to join $x$ to $y$ in time $t$ is
$$h_t(x,y)=\inf_\gamma \Act(\gamma)=\inf_\gamma \int_0^tL(\gamma(s),\dot\gamma(s))\, ds,$$
where the infimum is taken over all absolutely continuous curves $\gamma: [0,t]\to M$, 
with $\gamma(0)=x$ and $\gamma(t)=y$.

A minimizer (for $L$) is an absolutely continuous curve $\gamma: [a,b]\to M$ such that
$$\Act(\gamma)=h_{b-a}(\gamma(a),\gamma(b)).$$

Tonelli's theorem \cite{butgiahil, ClarkeSIAM, FatLecNot} states that, for every $a<b\in \R$ and every $x,y\in M$, there exists a minimizer
$\gamma:[a,b]\to M$, with $\gamma(a)=x,\gamma(b)=y$. All minimizers are as smooth as $L$.

Moreover, all minimizers are extremals, i.e. they satisfy the Euler-Lagrange equation given, in local coordinates, by
\begin{equation}\label{EL}
\frac{d}{dt}\left[\frac{\partial L}{\partial v}(\gamma(s),\dot\gamma(s))\right]=
\frac{\partial L}{\partial x}(\gamma(s),\dot\gamma(s)).
\end{equation}
As is well-known the 2$^{\rm nd}$ order ODE \eqref{EL} on $M$ yields a 1$^{\rm st}$ order ODE on $TM$
which generates a flow $\phi^L_t$ on $TM$ called the Euler-Lagrange flow.
A curve $\gamma:[a,b]\to M$ is an extremal (i.e. satisfies \eqref{EL}) if and only if its speed curve
$s\mapsto (\gamma(s),\dot\gamma(s))$ is (a piece of) an orbit of the Euler-Lagrange flow $\phi^L_t$.

An absolutely continuous curve $\gamma:[a,b]\to M$ is called a \emph{local} minimizer (for $L$) 
if there exists a neighborhood $U$ of $\gamma([a,b])$ in $M$ such that for any other absolutely continuous curve
$\delta:[a,b]\to U$, with $\delta(a)=\gamma(a)$ and $\delta(b)=\gamma(b)$, we have $\Act(\delta)\geq \Act(\gamma)$.
The regularity part of Tonelli's theorem implies that such a local minimizer $\gamma:[a,b]\to M$ 
is as smooth as $L$ and satisfies the Euler-Lagrange equation \eqref{EL}. Therefore its speed curve
$s\mapsto (\gamma(s),\dot\gamma(s))$ is (a piece of) an orbit of the Euler-Lagrange flow $\phi^L_t$.
\begin{exe}\label{Exe1} The simplest Tonelli Hamiltonian $H_g=T^*M\to \R$ 
on the complete Riemannian manifold $(M,g)$ is given by
$$H_g(x,p)=\frac12\lVert p\rVert^2_x.$$
Its associated Lagrangian $L_g:TM\to\R$  is given by:
$$L_g(x,v)=\frac12\lVert v\rVert^2_x.$$
If $\gamma:[a,b]\to M$ is a curve,  we will denote by $\Act_g(\gamma)$ its action for this Lagrangian $L_g$.
$$\Act_g(\gamma)=\frac12\int_a^b\lVert \dot\gamma(s)\rVert^2_{\gamma(s)}\,ds.$$
\end{exe}
\begin{lem}\label{MinGeo} For any curve $\gamma:[a,b]\to M$ we have
\begin{equation}
\begin{aligned}
\Act_g(\gamma)&\geq  \frac{\ell_g(\gamma)^2}{2(b-a)}\\
&\geq \frac{d^2(\gamma(a),\gamma(b))}{2(b-a)},
\end{aligned}
\label{InegCS}
\end{equation}
with equality if and only if $\gamma$ is a minimizing $g$-geodesic. 

Therefore, for every $t>0$ and every  $x,y\in M$, we have
\begin{equation}\label{MinActRiem}
h^g_t(x,y)=\frac{d(x,y)^2}{2t}.
\end{equation}
Moreover, a curve $\gamma:[a,b]\to M$ is $L_g$-minimizing if and only if it is a minimizing
geodesic. 
\end{lem}
\begin{proof}
Indeed, using the Cauchy-Schwarz inequality, we have
\begin{align*}
 \ell_g(\gamma)^2
&=\left(\int_a^b\lVert \dot\gamma(s)\rVert_{\gamma(s)}\,ds\right)^2\\
&\leq (b-a)\int_a^b\lVert \dot\gamma(s)\rVert^2_{\gamma(s)}\,ds\\
&=2(b-a)\Act_g(\gamma),
\end{align*}
with equality if and only if $\lVert \dot\gamma(s)\rVert_{\gamma(s)}$ is constant.
Since $d(\gamma(a),\gamma(b))\leq \ell_g(\gamma)$, this establishes \eqref{InegCS}.
Moreover, we have equality in \eqref{InegCS} if and only if  $d(\gamma(a),\gamma(b))= \ell_g(\gamma)$
and $\lVert \dot\gamma(s)\rVert_{\gamma(s)}$ is constant, which is precisely the case when 
$\gamma$ is a minimizing $g$-geodesic.

Since the Riemannian manifold $(M,g)$ is complete, given $t>0$ and   $x,y\in M$, we can find 
a minimizing $g$-geodesic $\gamma:[0,t]\to M$, with $\gamma(0)=x$ and $\gamma(t)=y$. 
From the first part of the Lemma, this proves \eqref{MinActRiem}.

The last part of the Lemma follows from the first two.
 \end{proof}
\begin{exe}\label{Exe1bis} If $(M,g)$ is a Riemannian  manifold, we define the Riemannian manifold $(M\times M,g\times g)$ as $M\times M$
with metric $g\times g$, i.e.
$$g\times g(x_1,x_2, v_1,v_2)=g(x_1,v_1)+g(x_2,v_2),$$
where we used the identification $T_{(x_1,x_2)}(M\times M)=T_{x_1}M\times T_{x_2}M$.
If $(M,g)$ is complete so is $(M\times M,g\times g)$.

Therefore $H_{g\times g}=T^*(M\times M)\to \R$ 
 is given by
$$H_{g\times g}(x,p)=\frac12\lVert p\rVert^2_x.$$
Its associated Lagrangian $L_{g\times g}:TM\to\R$  is given by:
$$L_{g\times g}(x_1,x_2, v_1,v_2)=\frac12\lVert v_1\rVert^2_{x_1}+\frac12\lVert v_2\rVert^2_{x_2}.$$
A curve $\Gamma:[a,b]\to M\times M$  is nothing but a pair of curves 
$\gamma_1,\gamma_2:[a,b]\to M$, such that $\Gamma(s)=(\gamma_1(s),\gamma_2(s))$, for all $s\in [a,b]$. We will denote
this identification by $\Gamma=(\gamma_1,\gamma_2)$.
\end{exe}
\begin{lem}\label{PropProD}  If $\Gamma=(\gamma_1,\gamma_2):[a,b]\to M\times M$ is a curve in $M\times M$,
its length  is
$$\ell_{g\times g}(\Gamma)=\ell_{g\times g}(\gamma_1,\gamma_2)
=\int_a^b\sqrt{\lVert \dot\gamma_1(s)\rVert^2_{\gamma_1(s)}
+\lVert \dot\gamma_2(s)\rVert^2_{\gamma_2(s)}}\,ds.$$
The action   of $\Gamma$ is given by
\begin{equation}\label{ActProd}
\begin{aligned}\Act_{g\times g}(\Gamma)=\Act_{g\times g}(\gamma_1,\gamma_2)&=\frac12\int_a^b\lVert \dot\gamma_1(s)\rVert^2_{\gamma_1(s)}
+\frac12\lVert \dot\gamma_2(s)\rVert^2_{\gamma_2(s)}\,ds\\
&=\Act_{g}(\gamma_1)+\Act_{g}(\gamma_2).
\end{aligned}
\end{equation}
The $g\times g$-distance in $M\times M$ is given by
$$d((x_1,x_2),(y_1,y_2))=\sqrt{d(x_1,y_1)^2+ d(x_2,y_2)^2}.$$
Therefore, since $h^{(g,g)}_t((x_1,x_2),(y_1,y_2))=d((x_1,x_2),(y_1,y_2))^2/(2t)$, we get
$$h^{(g,g)}_t={\frac{d(x_1,y_1)^2+d(x_2,y_2)^2}{2t}}.$$

This equality implies that the $g\times g$-geodesics (resp.\ $g\times g$-minimizing geodesics) in $M\times M$,
which are also the $L_{g\times g}$-extremals (resp.\ $L_{g\times g}$-minimzers), are the curves 
$\Gamma=(\gamma_1,\gamma_2):[a,b]\to M$
 such that $\gamma_1$ and $\gamma_2$ are both $g$-geodesics  
 (resp.\ $g$-minimizing geodesics).
\end{lem}

Most of the proof of Lemma \ref{PropProD} is similar to or uses Lemma
\ref{MinGeo}.

In order to give the connection between action and viscosity solutions it is convenient to introduce the following (obvious)
definition.
\begin{defn}[Graph of a curve] If $\gamma:[a,b]\to M$, we define its graph $\Graph(\gamma)\subset \R\times M$ as 
$$ \Graph(\gamma)=\{(t,\gamma(t)\mid t\in[a,b]\}.$$
\end{defn}
We now recall the definition of domination for a function.
\begin{defn} If $U: O\to [-\infty, +\infty]$ is a function defined on the subset $O\subset \R\times M$, we say that
$U$ is dominated by $L$ (on $O$) if for every absolutely continuous curve $\gamma:[a,b]\to M$, with $\Graph(\gamma)\subset O$ and $\Act(\gamma)=\int_a^bL(\gamma(s),\dot\gamma(s))\, dt<+\infty$,
we have
\begin{equation}\label{Dom1}
U(\gamma(b),b)\leq \Act(\gamma)=\int_a^bL(\gamma(s),\dot\gamma(s))\, dt
+U(\gamma(a),a),
\end{equation}
\end{defn}
Note that the right hand side of \eqref{Dom1} makes always sense since we insisted that $\Act(\gamma)<+\infty$.
\begin{rem} 1) We used the inequality \eqref{Dom1} rather than the usual
\begin{equation}\label{Dom2}U(\gamma(b),b)-U(\gamma(a),a)\leq \Act(\gamma)
=\int_a^bL(\gamma(s),\dot\gamma(s))\, dt.
\end{equation}
since it will be convenient to consider real valued functions with possibly infinite values.
Of course when $U$ is finite valued \eqref{Dom1} and \eqref{Dom2} are equivalent.

2) If $U:I\times M\to \R$, where $I\subset \R$ is an interval, then all curves defined on a subinterval $[a,b]\subset
I$ have their graph included in $I\times M$. Therefore, in such a case, the function $U$  is dominated by $L$ 
if and only if
\begin{equation}\label{Dom3}
U(t',x')-U(t,x)\leq h_{t'-t}(x,x'),\text{ for every $x,x'\in M$, and every $t<t'\in I$.}
\end{equation}
\end{rem}
A first connection between action and viscosity solutions is given by the following proposition, see \cite{Ishii} or
\cite{FatNC} for a proof.
\begin{prop}\label{domination1} Let $H:T^*\/M\to\R$ be a Tonelli Hamiltonian. Assume 
$ V:O\to \R$ is a continuous function, where $O\subset R\times M$ is an open subset of $\R\times M$.
Then $V$ is a viscosity subsolution of 
\begin{equation}\label{HJE00}
\partial_t V+H(x,\partial_x V)=0,
\end{equation}
on $ O$ if and only if it is dominated by $L$ on $O$,
where $L$ is the Lagrangian associated to $H$.
\end{prop}
\subsection{Calibrated curves, backward characteristics}
\begin{defn}[Calibrated curve]
Let $U:O\to [-\infty,+\infty]$ be a function. An absolutely curve  $\gamma: [a,b]\to M$ is said to be 
$U$-calibrated (for the Lagrangian $L$) if $\Graph(\gamma)\subset M$, 
its action $ \Act(\gamma)=\int_a^bL(\gamma(s), \dot \gamma(s))\, ds$ is finite, and
\begin{equation}\label{CalCond}
U(b,\gamma(b))=U(a,\gamma(a))+\Act(\gamma)=U(a,\gamma(a))+\int_a^bL(\gamma(s), \dot \gamma(s))\, ds.
\end{equation}
\end{defn}
Again we used \ref{CalCond}, rather than the more usual $U(b,\gamma(b))-U(a,\gamma(a))=\Act(\gamma)$,
because we would like to allow possibly infinite values for $U$.
\begin{prop}\label{CONCAT} Suppose $U:[O\to [-\infty,+\infty]$ is a function defined on the subset $O$. 

 If the absolutely continuous curve $\gamma: [a,b]\to M$ is
piecewise $U$-calibrated, then it is $U$-calibrated.
\end{prop}
Of course, the curve $\gamma:[a,b]\to M$ is said to be piecewise calibrated if we can find a finite sequence
$a=t_0<t_1<\cdots<t_\ell=b$ such that each restriction $\gamma|[t_i,t_{i+1}], i=0,\dots,\ell-1$ is $U$-calibrated.
\vskip .2cm
The notion of calibrated curve is useful when $U$ is a viscosity subsolution as can be seen from the following 
well-known proposition.
\begin{prop}\label{calcurprop} Suppose the function $U:O\to \R$ is continuous and a viscosity subsolution,
on $O$,
of the evolutionary Hamilton-Jacobi equation \eqref{HJE0}. If $\gamma:[a,b]\to M$ is $U$-calibrated,
with either $U(a,\gamma(a))$ or $U(b,\gamma(b))$ finite, we have:
\begin{enumerate}[(1)]
\item For every $t\in[a,b]$, the value $U(t,\gamma(t))$ is finite.
\item The restriction $\gamma|[a',b']$ is also $U$-calibrated, for any subinterval 
$[a',b']\subset [a,b]$.
\item the curve $\gamma$
 is a local minimizer. Hence it is as smooth as $L$ and a solution of the Euler-Lagrange equation
\eqref{EL}.
\end{enumerate}
\end{prop}
For the next proposition we will use the notion of upper
and lower differentials, see \cite{Barles, BarCap, CanSin, FatLecNot, FatEcos} for more details on this notion and its relationship with viscosity solutions.
\begin{nota} If $w:N\to \R$ is a function on the manifold $N$ and $n\in N$, the set of upper-differentials 
(resp.\ lower-differentials) of $w$ at
$n$ is denoted by $D^+w(n)\subset T_n^*\/N$ (resp.\ $D^-w(n)\subset T_n^*\/N$).
\end{nota}
\begin{prop}\label{DiffProp} Suppose the function $U:O\times M\to \R$, defined on the open subset $O
\subset \R\times M$ is continuous and a viscosity subsolution
of the evolution Hamilton-Jacobi equation \eqref{HJE0} on $O$. Assume that $\gamma: [a,b]\to M$,
is a $U$-calibrated curve. 
\begin{enumerate}[\rm (i)]
\item For every $t\in ]a,b]$, we have
$$\left (-H\left(\gamma(t),\partial_v L(\gamma(t), \dot\gamma(t))\right), \partial_v L(\gamma(t), \dot\gamma(t))\right)\in D^+U(t,\gamma(t)),$$
where $D^+U(t,\gamma(t))\subset \R\times T_{\gamma(t)}^*\/N \cong T_{(t,\gamma(t))}^*(\R\times N)$.
\item For every $t\in [a,b[$, we have
$$\left (-H\left(\gamma(t),\partial_v L(\gamma(t), \dot\gamma(t))\right), \partial_v L(\gamma(t), \dot\gamma(t))\right)\in D^-U(t,\gamma(t)),$$
where $D^-U(t,\gamma(t))\subset T_{(t,\gamma(t))}^*(\R\times N)\cong \R\times T_{\gamma(t)}^*\/N$.
\item In particular, if $\partial_x U$ exists at $(t,\gamma(t))$ with $t\in [a,b]$, then 
$$\partial_x U(t,\gamma(t))= \partial_v L(\gamma(t), \dot\gamma(t)).$$
\item If  $\partial_t U$ exists at $(t,\gamma(t))$ with $t\in [a,b]$, then
$$\partial_t U(t,\gamma(t))+H\left(\gamma(t), \partial_v L(\gamma(t), \dot\gamma(t))\right)=0.$$
\item Moreover, the function $U$ is indeed differentiable at every point $(t,\gamma(t))$ with $t\in ]a,b[$.
\end{enumerate}
\end{prop}

\begin{defn}[Backward characteristic] Let $U:O\to [-\infty,+\infty]$ be a function where $O\subset R\times M$. 
A backward $U$-characteristic ending at $(t,x)\in ]0,T[\times M$ is 
a $U$-calibrated curve $\gamma:[a,t]\to M$ with $a<t$ and $\gamma(t)=x$.

More generally, a curve $\gamma:[a,t]\to M$ is called  a backward $U$-characteristic if
 it is a backward characteristic ending at $(t,\gamma(t))$.
\end{defn}
Again this notion of backward characteristic is useful only when $U$ is at least 
a viscosity subsolution of the evolution Hamilton-Jacobi equation \eqref{HJE0}. 
Using that backward characteristics are extremals, the following  proposition is a consequence of  Proposition \ref{DiffProp}.
\begin{prop}\label{DiffProp2}  Suppose the function $U:O\to \R$, where $O$ is an open subset 
of $\R\times M$, is continuous and a viscosity subsolution
of the evolutionary Hamilton-Jacobi equation \eqref{HJE0} on $O$. 
 If $\gamma:[a,t]\to M$,  is a backward $U$-characteristic,
then $U$ is differentiable at  $(s,\gamma(s))$ for all $s\in ]a,t[$.

Moreover, if $U$ is differentiable at $(t,\gamma(t))$ then $\gamma$ is the unique backward characteristic 
(for $U$) ending at $(t,\gamma(t))$, and we have
\begin{gather*}
\partial_x U(t,\gamma(t))= \partial_v L(\gamma(t), \gamma(t))\\
\partial_t U(t,\gamma(t))+H\left(\gamma(t), \partial_x U(t,\gamma(t))\right)=0.
\end{gather*}
\end{prop} 

\subsection{The negative Lax-Oleinik semi-group and the negative Lax-Oleinik evolution}
In fact, viscosity solutions which are continuous are always given by the negative Lax-Oleinik
evolution as we now recall.

Once the minimal action is defined, we can introduce the negative Lax-Oleinik semi-group.

If $u:M\to [-\infty,+\infty]$ is a function and $t>0$, the function $T^-_tu:M\to [-\infty,+\infty]$ is defined by
$$
T^-_tu(x)=\inf_{y\in M}u(y)+h_t(y,x),$$
We also set $T^-_0u=u$. The   negative Lax-Oleinik semi-group is
$T^-_t,t\geq 0$. 

It is convenient to define $\hat u:[0,+\infty[\times M\to[-\infty,+\infty]$ by 
$$\hat u(t,x)=T^-_tu(x).$$
This function $\hat u$ is called the negative Lax-Oleinik evolution of $u$.
\begin{exe}\label{Exe2} By \eqref{MinActRiem}, for the Hamiltonian $H_g:T^*M\to \R$,
and Lagrangian $L_g:TM\to \R$, defined in Example \ref{Exe1}, we have
$$h^g_t(x,y)=\frac{d(x,y)^2}{2t}.$$
Therefore, the associated  negative Lax-Oleinik semi-group $T^{g-}_t$ is defined, when $t>0$, by
$$T^{g-}_tu(x)=\inf_{y\in M}u(y)+\frac{d(y,x)^2}{2t},$$
for $u:M\to [-\infty,+\infty]$.

If $C\subset M$, we define its (modified) characteristic function $\chi_C: M\to \{0,+\infty\}$ by
$$\chi_C(x)=
\begin{cases}
0, \text{ if $x\in C$,}\\
+\infty, \text{ if $x\notin C$.}
\end{cases}
$$
Therefore its negative Lax-Oleinik evolution $\hat \chi_C$, for the Lagrangian $L_g$,  is defined, for $t>0$, by

\begin{align*}
\hat \chi_C(t,x)&=\inf_{y\in M}\chi_C(y)+\frac{d(y,x)^2}{2t}\\
&=\inf_{c\in C}\frac{d(c,x)^2}{2t}.
\end{align*}
Since $d_C:M\to \R$, the distance function to $C$, is given by
\begin{equation}\label{DefdC}d_C(x)=\inf_{c\in C}d(c,x),
\end{equation}
for $t>0$, we obtain
\begin{equation}\label{ValHatXi}
\hat \chi_C(t,x)=\frac{d_C(x)^2}{2t}.
\end{equation}
Note that $d_C=d_{\bar C}$, where $\bar C$ is the closure of $C$ in $M$. Hence, we get
$\hat \chi_C=\hat \chi_{\bar C}$ on $]0,+\infty[\times M$. Therefore, to study of the properties 
of $\hat \chi_C$ on $]0,+\infty[\times M$, we can always assume that $C$ is a closed subset 
of $M$.
\end{exe}
\begin{lem}\label{CacarCaldC} Suppose $C$ is a closed subset  of the complete Riemannian manifold $(M,g)$.
A curve $\gamma:[a,b]\to M$, with $a>0$, is $\hat \chi_C$-calibrated if and only if it is a minimizing
$g$-geodesic and
$$\frac{d^2_C(\gamma(b))}{2b}-\frac{d^2_C(\gamma(a))}{2a}=\frac{d^2(\gamma(b),\gamma(a))}{2(b-a)}.$$
A curve $\gamma:[0,b]\to M$ is $\hat \chi_C$-calibrated if and only if it is a minimizing
$g$-geodesic, with $\gamma(0)\in C$ and
$$d_C(\gamma(b))=d(\gamma(b),\gamma(0)).$$
\end{lem}
\begin{proof} Suppose that the curve $\gamma:[a,b]\to M$, with $a>0$, is 
$\hat \chi_C$-calibrated. It must be a minimizing geodesic. 
We now recall, see \ref{InegCS}, that for a minimizing $g$-geodesic  $\gamma:[a,b]\to M$, we have
$$\Act_g(\gamma)=\frac{d^2(\gamma(b),\gamma(a))}{2(b-a)}.$$
 Therefore, using \eqref{ValHatXi},
we see that calibration for a minimizing $g$-geodesic $\gamma:[a,b]\to M$, with $a>0$, is equivalent to
 $$\frac{d^2_C(\gamma(b))}{2b}=\frac{d^2_C(\gamma(a))}{2a}+\frac{d^2(\gamma(b),\gamma(a))}{2(b-a)}.$$

This finishes to prove the first part with $a>0$.

The curve $\gamma:[0,b]\to M$ is $\hat \chi_C$-calibrated if and only if 
$$\frac{d^2_C(\gamma(b))}{2b}=\chi_C(\gamma(0))+ \Act_g(\gamma).$$
Since the left hand side is finite and $ \Act_g(\gamma)>-\infty$, this is equivalent to $\gamma(0)\in C$
and $\frac{d^2_C(\gamma(b))}{2b}=\Act_g(\gamma)$.

Assume $\gamma:[0,b]\to M$, with $\gamma(0)\in C$, is $\hat \chi_C$-calibrated.
As we saw above we have $ \Act_g(\gamma)= {d^2_C(\gamma(b))}/{2b}$
Since $d_C(\gamma(b))\leq d(\gamma(b),\gamma(0))$, and $\Act_g(\gamma)\geq 
d^2(\gamma(b),\gamma(0))/(2b)$, we indeed get 
$d_C(\gamma(b))=d(\gamma(b),\gamma(0))$.

Conversely, if  $\gamma:[0,b]\to M$ is  a minimizing
$g$-geodesic, with $\gamma(0)\in C$ and $d_C(\gamma(b))=d(\gamma(b),\gamma(0))$,
we get 
$$\frac{d^2_C(\gamma(b))}{2b}=\frac{d^2(\gamma(b),\gamma(0))}{2b}=\Act_g(\gamma).$$
Since $\gamma(0)\in C$, by what we obtained above, this implies that
$\gamma$ is  $\hat \chi_C $-calibrated.
\end{proof}
We will now give the relationship between the Aubry set 
${\cal I}_\infty(\hat \chi_C)$ of $\hat\chi_C$--see Definition \ref{AubryEvol}-- and the Aubry set ${\cal A}^*(C)$ of the closed set $C$--see Definition \ref{AubryClosed}.
\begin{prop}\label{AubryDistProp}
 If $C$ is a closed subset  of the complete Riemannian manifold $(M,g)$, then we have
 $$ {\cal I}_\infty(\hat \chi_C)=]0,+\infty[\times (C\cup {\cal A}^*(C)).$$
\end{prop}
\begin{proof}
From Lemma \ref{CacarCaldC}, a curve $\gamma:[0,+\infty[\to M$ is $\hat \chi_C$-calibrated for the Lagrangian $L_g$, if and only if it is a minimizing $g$-geodesic satisfying
\begin{equation}\label{CondAubDist}
\gamma(0)\in C \text{ and } d_C(\gamma(t))=d(\gamma(t),\gamma(0)),\text{ for all $t> 0$}.
\end{equation}
Therefore, a constant curve with value in $C$ is $\hat \chi_C$-calibrating. This implies that $]0,+\infty[\times C\subset
{\cal I}_\infty(\hat \chi_C)$.

For a curve $\gamma:[0,+\infty[\to M$ and $\lambda>0$,  we define $\gamma_\lambda:[0,+\infty[\to M$ by
$$\gamma_\lambda(t)=\gamma(\lambda t).$$
Obviously, the curve $\gamma:[0,+\infty[\to M$ is a minimizing
$g$-geodesic that satisfies \eqref{CondAubDist},  if  and only $\gamma_\lambda$ is also a minimizing $g$-geodesic satisfying \eqref{CondAubDist}.

Assume now that $(t,y)\in {\cal I}_\infty(\hat \chi_C)$, with $y\notin C$. We can find a $\hat \chi_C$-calibrated curve $\gamma:[0,+\infty[\to M$, with $\gamma(t)=y$. Since $\gamma(0)\in C$,
the $g$-geodesic $\gamma$ is not constant. Since geodesics are parametrized proportionally to arc-length,  we can find $\lambda$ such that $\gamma_\lambda$ is parameterized by arc length.
As we saw above, the curve $\gamma_\lambda$ is a minimizing $g$-geodesic satisfying
\eqref{CondAubDist}. Since, the $g$-geodesic $\gamma_\lambda$ is minimizing and parametrized by arc-length, we get
$ d_C(\gamma_\lambda(s))=d(\gamma_\lambda(s),\gamma_\lambda.(0))=s$.
By Definition \ref{AubryClosed}, this means that $y\in {\cal A}^*(C)$. Hence ${\cal I}_\infty(\hat \chi_C)\subset]0,+\infty[\times (C\cup {\cal A}^*(C))$.

It remains to show that $]0,+\infty[\times {\cal A}^*(C)
\subset {\cal I}_\infty(\hat \chi_C)$. Suppose $y\in {\cal A}^*(C)$.
By Definition \ref{AubryClosed}  and Remark \ref{RemdCcal}, we can find a minimizing 
$g$-geodesic
$\gamma:[0,+\infty[\to M$ parameterized by arc-length such that
$d_C(\gamma(t))=d(\gamma(0),\gamma(t))$ and $y=\gamma(t_0)$ 
for some $t_0>0$. Therefore, for every $\lambda>0$, the curve
$\gamma_{\lambda}$ is $\hat \chi_C$-calibrated. Therefore, we get
$(t/\lambda,\gamma_{\lambda}(t))\in {\cal I}_\infty(\hat \chi_C)$,
for every $t>0$. Since $\gamma_{\lambda}(t_0/\lambda)=\gamma(t_0)=y$,
and $\lambda>0$ is arbitrary, we conclude that
$(t,y)\in {\cal I}_\infty(\hat \chi_C)$, for every $t>0$.
Hence $]0,+\infty[\times {\cal A}^*(C)\subset 
{\cal I}_\infty(\hat \chi_C)$.
\end{proof}
\begin{exe}\label{Exe2bis} We specialize the previous case to the diagonal 
$\Delta_M=\{(x,x)\mid  x\in M\}\subset M\times M$ of the Riemannian manifold $(M\times M,g\times g)$.
We will show that
\begin{equation}\label{CardDelta}
\hat \chi_{\Delta_M}(x,y)=\frac{d_{\Delta_M}(x,y)^2}{2t}=\frac{d(x,y)^2}{4t}.
\end{equation}
The left hand side equality follows from \eqref{MinActRiem}, with $C=\Delta_M$.
It remains to compute   $d_{\Delta_M}(x,y)$. We will need the following simple well-known Lemma:
\begin{lem}\label{IneqNum} For  any $k\geq 0$ and any $a,b$ with  $a+b\geq k$, we have
$ a^2+b^2\geq k^2/2$, with equality if and only if $a=b=k/2$.
\end{lem}
We have
$$d_{\Delta_M}(x,y)^2=\inf_{c\in M}d((x,y),(c,c))^2= \inf_{c\in M}d(x,c)^2 +d(y,c)^2$$
We now note that $d(x,c)+d(y,c)\geq d(x,y)$, and, since $M$ is complete, that we can find
$c_0\in M$, with $d(x,c_0)=d(y,c_0)=d(x,y)/2$.  Therefore, by
 Lemma \ref{IneqNum} we get
 $$d_{\Delta_M}(x,y)^2=\frac{d(x,y)^2}{2},$$
which yields  the right hand side equality in 
\eqref{CardDelta}.
\end{exe}
\begin{lem}\label{CaracCaldDelta} 
	Let $\Gamma:[0,t]\to M\times M$ be a curve in 
	$M\times M$, with $\Gamma(s)=(\gamma_1(s), \gamma_2(s))$,
	with $\gamma_1,\gamma_2:[0,t]\to M$.  The curve $\Gamma$ is 
	$\hat \chi_{\Delta_M}$-calibrated if and only if $\gamma_1$ and
	 $\gamma_2$ are both minimizing $g$-geodesics, with
	 $\gamma_1(0)=\gamma_2(0)$ and
	 $$d(\gamma_1(0),\gamma_1(t))=d(\gamma_2(0),\gamma_2(t))=
	 \frac{d(\gamma_1(t),\gamma_2(t))}{2}.$$
	 Therefore, if $\Gamma=(\gamma_1, \gamma_2):[0,t]\to M\times M$
	  is such a  $\hat \chi_{\Delta_M}$-calibrated curve, 
	  then the curve $\gamma:[-t,t]\to M$ defined by
	 \begin{equation}\label{DefConcatGamma}
	 \gamma(s)=\begin{cases}
	 \gamma_1(-s),\text{ if $s\in [-t,0]$,}\\
	 \gamma_2(s),\text{ if $s\in [0,t]$,}
	 \end{cases}
	 \end{equation}
 is a minimizing $g$-geodesic.
 
Conversely, if  $\gamma:[-t,t]\to M$,  is a minimizing $g$-geodesic,
then $\Gamma=(\gamma_1, \gamma_2):[0,t]\to M\times M$ defined,
for  $s\in [0,t]$, by
\begin{equation}
\begin{aligned}
\gamma_1(s)&=\gamma(-s)\\
\gamma_2(s)&=\gamma(s)
\end{aligned}
\label{DefGam}
\end{equation}

is $\hat \chi_{\Delta_M}$-calibrated.
\end{lem}
\begin{proof} From Lemma \ref{CacarCaldC} applied to the closed set $\Delta_M$ of the Riemannian manifold 
$(M\times M,g\times g)$, we get that $\Gamma=(\gamma_1, \gamma_2):[0,t]\to M\times M$ is 
$\hat \chi_{\Delta_M}$-calibrated curve if and only if  $\Gamma$ is a minimizing $g\times g$-geodesic 
(which is equivalent to the fact both $\gamma_1$ and $\gamma_2$ are minimizing $g$-geodesics) 
with $\Gamma(0)\in\Delta_M$ (which is equivalent to  $\gamma_1(0)=\gamma_2(0)$) and
$$d_{\Delta_M}(\Gamma(t))=d(\Gamma(0),\Gamma(t)).$$
By \eqref{CardDelta} in Example \ref{Exe2bis}, the last identity is in turn equivalent to
\begin{equation}\label{TOTO333}
\frac{d(\gamma_1(t),\gamma_2(t))}{\sqrt2}=
\sqrt{d(\gamma_1(0),\gamma_1(t))^2+d(\gamma_2(0),\gamma_2(t))^2}.
\end{equation}
Since, we already know that $\gamma_1(0)=\gamma_2(0)$, we get
$$d(\gamma_1(0),\gamma_1(t))+d(\gamma_2(0),\gamma_2(t))\geq 
d(\gamma_1(t),\gamma_2(t)).$$
	By  Lemma \ref{IneqNum}, we obtain that \eqref{TOTO333}
	is equivalent to
	$$d(\gamma_1(0),\gamma_1(t))=d(\gamma_2(0),\gamma_2(t))=
	\frac{d(\gamma_1(t),\gamma_2(t))}{2}.$$
	This finishes to prove the characterization of $\hat \chi_{\Delta_M}$-calibrated curves $\Gamma:[0,t]\to M\times M$.
 
	Assume now that $\Gamma=(\gamma_1,\gamma_2):[0,t]
	\to M\times M$. Since $\gamma_1(0)=\gamma_2(0)$, the curve 
	$\gamma$ in \eqref{DefConcatGamma}
	 is well defined and
	$$\ell_g(\gamma)= \ell_g(\gamma_1)+\ell_g(\gamma_2).$$
	As we showed above $\gamma_1$ and $\gamma_2$ are both minimizing $g$-geodesics, with 
	$$d(\gamma_1(0),\gamma_1(t))=d(\gamma_2(0),\gamma_2(t))=
	\frac{d(\gamma_1(t),\gamma_2(t))}{2}.$$
	Therefore
	\begin{align*}
		\ell_g(\gamma)&=d(\gamma_1(0),\gamma_1(t))+d(\gamma_2(0),\gamma_2(t))\\
	&=d(\gamma_1(t),\gamma_2(t))\\
	&=d(\gamma(-t),\gamma(t)),
	\end{align*}
	which implies that $\gamma$ is a $g$-minimizing geodesic.
	
	Conversely, if  $\gamma:[-t,t]\to M$,  is a minimizing $g$-geodesic,
	and $\gamma_1, \gamma_2$ are defined by \eqref{DefGam},
	then they are minimizing $g$-geodesics and
	$$d(\gamma_2(0),\gamma_2(t))=
	d(\gamma_1(0),\gamma_2(t))=	\frac{d(\gamma(-t),\gamma(t))}{2},$$
	since a geodesic is parametrized proportionally to arc-length.
	Hence $\Gamma=(\gamma_1, \gamma_2)$	is $\hat \chi_{\Delta_M}$-calibrated.
\end{proof}
The following theorems were obtained in \cite{FatNC}. We will sketch a proof of the first one in Appendix \ref{Appendice}.
The proof of the second one uses the methods of Appendix A together with an appropriate form 
of the maximum principle.
\begin{thm}\label{PropHatu} Assume $H:T^*\/M\to \R$ is a Tonelli Hamiltonian. Let $u:M\to [-\infty,+\infty]$ be a function and denote
by $\hat u:[0,+\infty[\times M\to [-\infty,+\infty]$ its negative Lax-Oleinik evolution. 
If $\hat u$  is finite-valued on $]0,t_0[\times M$, then it is continuous
and even locally semiconcave on $]0,t_0[\times M$.

Moreover, the negative Lax-Oleinik evolution is a viscosity solution, on $]0,t_0[\times M$, of the evolutionary
Hamilton-Jacobi equation \eqref{HJE0}.
\end{thm}
\begin{thm}\label{UniquenessUC} Assume $H:T^*\/M\to \R$ is a Tonelli Hamiltonian. If the  continuous function 
$U:]0,t_0[ \times M\to \R$ 
is a viscosity solution of the evolutionary Hamilton-Jacobi equation
$$
\partial_tU+ H(x,\partial_xU)=0
$$
on  $]0,t_0[ \times M$,  then $U=\hat u$ on $]0,t_0[ \times M$ for some function $u:M\to [-\infty,+\infty]$.
\end{thm}
\begin{rem} Note that in the previous two theorems no assumption (continuity or else) is made on $u$. 
However, the finiteness assumption of $\hat u$ on $]0,t_0[ \times M$, puts some strong restrictions on the behaviour of $u$. 
For example it is not difficult to check that this finiteness assumption implies that $u$ is bounded below by a continuous function, even when $M$ is not compact.

In fact,  if $(t,x)\in M\times ]0,t_0[$, since $\hat u(t,x)=\inf_{y\in M}u(y)+h_t(y,x)$, then $u(y)\geq \hat u(t,x)-h_t(y,x)$.
But  the function of $y\mapsto \hat u(t,x)-h_t(y,x)$ is continuous on $M$.

In particular, when $M$ is compact the finiteness assumption  on $\hat u$ is equivalent to $u$ 
being bounded from below and not equal to $+\infty$ everywhere. 

Even for $M$ non-compact, if $u$  is bounded from below not equal to $+\infty$ everywhere, 
its negative Lax-Oleinik evolution $\hat u$ is finite everywhere.
\end{rem}

\vskip .2cm
It follows from  Theorem \ref{UniquenessUC} that Theorem \ref{LocContr} for $O={}]0,t_0[\times TM$
is a particular case of Theorem  \ref{LocContrGen} below, which states
that $\Sigma_{t_0}(\hat u)$ is locally contractible for any $u:M\to [-\infty,+\infty]$ such that $\hat u$ is finite on
$]0,t_0[\times M$.

We now recall some more facts on $\hat u$ obtained, for example, in \cite{FatNC}.
\begin{prop}\label{hatu-} For every function $u:M\to [-\infty,+\infty]$, we have $\hat u=\hat u_-$ 
on $]0,+\infty[\times M$, with 
$u_-: M\to [-\infty,+\infty]$ the lower semi-continuous regularization of $u$ given by
$$u_-(x)=\liminf_{y\to x}u(y)=\sup_V\inf_{y\in V}u(y),$$
where the supremum is taken over all neighborhoods $V$ of $x$.
\end{prop}
We recall that   $u_-$ is the largest lower semi-continuous function which is $\leq u$.
\begin{exe}\label{Exe3}
If $C\subset M$, it is not difficult to see that the lower semi-continuous regularization of the characteristic 
function $\chi_C$ is precisely the characteristic function $\chi_{\bar C}$, where $\bar C$ is the closure of $C$ in $M$.
\end{exe}
As a consequence of Proposition \ref{hatu-}, without loss of generality, we can assume that the function $u$ is lower semi-continuous when we consider  properties of $\hat u$ away from $\{0\}\times M$.

This is quite convenient since, as shown in \cite{FatNC}, see also the Remark \ref {ExistenceCalCurve} in the Appendix, 
it allows to have backward characteristics.
\begin{prop}\label{ExistenceBackCar} If the function $u:M\to [-\infty,+\infty]$ is lower semi-continuous, and its negative Lax-Oleinik evolution 
$\hat u$ is finite on $]0,t_0[\times M$, with $t_0\in ]0,+\infty]$, then, for every $(t,x)\in ]0,t_0[\times M$, we can find a backward $\hat u$-charac\-te\-ris\-tic $\gamma :[0,t]\to M$ 
with $\gamma(t)=x$.
\end{prop} 

The existence of backward characteristics at every points allows to
give a criterion for differentiability, which complements both Propositions \ref{DiffProp} and \ref{DiffProp2}
\begin{prop}\label{DiffCriterion} Assume that the function $u:M\to [-\infty,+\infty]$ is lower semi-continuous, 
and  that its negative Lax-Oleinik evolution 
	$\hat u$ is finite on $]0,t_0[\times M$, with $t_0\in ]0,+\infty]$. Then, the function $\hat u$ is differentiable at a point $(t,x)\in ]0,t_0[\times M$,  if and only if there exists  a unique backward $\hat u$-charac\-te\-ris\-tic $\gamma :[0,t]\to M$ 
	with $\gamma(t)=x$.
\end{prop} 

\begin{exe}\label{Exe4} If $C$ is a closed subset of the complete Riemannian manifold $(M,g)$,  Proposition \ref{ExistenceBackCar} applied to 
the Lagrangian $L_g$ of Example  \ref{Exe1} and the function $\chi_C$ shows that for every $x\in M$
and every $t>0$, we can find a $g$-geodesic $\gamma: [0,t]\to M$ with $\gamma(0)\in C, \gamma(t)=x$ and
$$\frac{d_C(x)^2}{2t}=\int_0^t\frac12\lVert \dot\gamma(s)\rVert^2_{\gamma(s)}\,ds=\frac{d(\gamma(0),x)^2}{2t}.$$
This last equality is of course equivalent to $d_C(x)=d(\gamma(0),x)$.

Given that $(M,g)$ is a complete Riemannian manifold, this is of course equivalent to the well-known fact that the $\inf$ in the definition \eqref{DefdC} of $d_C(x)$
is always attained by a point $c\in C$, when $C$ is closed.

Moreover, by Proposition \ref{DiffCriterion}, we obtain that
$d_C$ is differentiable at $x\notin C$ if and only if there is a
unique (up to reparametrization) minimizing $g$-geodesic
$\gamma:[0,t]\to M$, with $\gamma(0)\in C,\gamma(t)=x$ and $d_C(x)=d(\gamma(0),x)$.
\end{exe}
\begin{exe}\label{Exe5} If we apply the previous example to the closed
	subset $\Delta_M$  of the complete Riemannian manifold $(M\times M,g\times g)$, by  Lemma  \ref{CaracCaldDelta}, we  first recover the well know fact that any pair of points $x,y$ in $M$ can be joined by a minimizing geodesic.
	
Moreover, by Proposition \ref{DiffCriterion} applied to $d^2_{\Delta_M}$, we obtain that
$d^2_{\Delta_M}$ is differentiable at $(x,y)\in M\times M$ if and only if there is a unique (up to reparametrization) 
minimizing $g$-geodesic in $M$ joining $x$ to $y$. 
Therefore, recalling from Definition \ref{DefUNU} that ${\cal U}(M,g)$  is the set of 
$(x,y)\in M\times M$ such that there exists a \emph{unique} minimizing $g$-geodesic between $x$  and $y$,
we obtain that the set of points in $M\times M$ where $d^2_{\Delta_M}$ is differentiable is
precisely ${\cal U}(M,g) $. 
Therefore, its complement $M\times M\setminus {\cal U}(M,g)={\cal NU}(M,g)$--the set of points $(x,y)\in M$ such that there exists at least two distinct minimizing  
$g$-geodesics between $x$ and $y$--is the set $\Sigma(d^2_{\Delta_M})$ of singularities of $d^2_{\Delta_M}$.
Since, the diagonal $\Delta_M$is contained in ${\cal U}(M,g)$, we conclude that
$$\Sigma^*(d_{\Delta_M})={\cal NU}(M,g).$$
\end{exe}

\subsection{Cut points and cut time function}
In this subsection, we will consider a lower semi-continuous function $u:M\to [-\infty,+\infty]$ such that 
$\hat u$ is finite on $]0,t_0[\times M$, where $t_0\in ]0,+\infty]$. 
By Theorem \ref{PropHatu}, on $]0,t_0[\times M$ the function $\hat u$ is locally semiconcave 
and a viscosity solution of the evolutionary Hamilton-Jacobi equation \eqref{HJE0}. Moreover, by 
Proposition \ref{ExistenceBackCar},  for every $(t,x)\in ]0,t_0[\times M$ 
we can find a backward $\hat u$-characteristic ending at $(t,x)$.

Recall that $\Sigma_{t_0}(\hat u)$ the set of singularities of $\hat u$ contained in $]0,t_0[\times M$.

We now introduce the set $\operatorname{Cut}_{t_0}(\hat u)\subset ]0,t_0[\times M$ of cut points of $\hat u$.
\begin{defn}[$\operatorname{Cut}_{t_0}(\hat u)$] The set $\operatorname{Cut}_{t_0}(\hat u)$ of cut points of $\hat u$ is the set of points $(t,x)\in ]0,t_0[\times M$
where no backward $\hat u$-characteristic ending at $(t,x)$ can be extended to a $\hat u$-calibrating curve
defined on $[0,t']$, with $t'>t$. 
\end{defn}
\begin{lem}\label{CritCut} Under the hypothesis above, the point $(t,x)\in ]0,t_0[\times M$ is in 
$\operatorname{Cut}_{t_0}(\hat u)$ if and only if 
for any  $\hat u$-calibrating curve $\delta:[a,b] \to M$, with 
$t\in [a,b]$ and $\delta(t)=x$, we have $t=b$.
\end{lem}
\begin{proof} In fact, if $\delta$ is as given above, and $\gamma:[0,t]\to M$ is a backward $\hat u$-characteristic ending at $(t,x)$, then the curve defined on $[0,b]\supset [0,t]$ which is equal  to $\gamma$ on $[0,t]$ and
$\delta$ on $[t,b]$ is continuous and piecewise calibrated on $[0,b]$. Therefore, by Lemma \ref{CONCAT}, it is 
a $\hat u$-calibrated curve extending $\gamma$. If $(t,x)\in \operatorname{Cut}_{t_0}(\hat u)$, then we must have
$t=b$. 
\end{proof}
\begin{prop} Under the hypothesis above, we have 
$$\Sigma_{t_0}(\hat u)\subset \operatorname{Cut}_{t_0}(\hat u)\subset ]0,t_0[ \times M\setminus {\cal I}_{t_0}(\hat u).$$ Moreover, the set
$\Sigma_{t_0}(\hat u)$ is dense in $\operatorname{Cut}_{t_0}(\hat u)$. Hence, we have
$$\Sigma_{t_0}(\hat u)\subset \operatorname{Cut}_{t_0}(\hat u)\subset \overline{\Sigma_{t_0}(\hat u)}.$$
\end{prop}
\begin{proof} The fact that $ \operatorname{Cut}_{t_0}(\hat u)\cap {\cal I}_{t_0}(\hat u)=\emptyset$ follows from Lemma
\ref{CritCut}.
Note now that for every point $(t,x)\in]0,t_0[ \times  M\setminus \operatorname{Cut}_{t_0}({\hat u})$, we can find a
$\hat u$-calibrated curve $\gamma:[0,t']\to M$ with  $t'>t$ and $\gamma(t)=x$. Therefore, by part (v) of Proposition
\ref{DiffProp}, the function $\hat u$ is differentiable at 
$(t,x)$. Hence we get $\Sigma_{t_0}(\hat u)\subset \operatorname{Cut}_{t_0}({\hat u})$.

We now prove the density of the inclusion $\Sigma_{t_0}(\hat u)\subset \operatorname{Cut}_{t_0}({\hat u})$. Suppose that
$V \subset  ]0,t_0[\times M$ is an open neighborhood of $(\bar t,\bar x)$ which contains 
no point of $\Sigma_{t_0}(\hat u)$.
Since $\hat u$ is semiconcave and is differentiable everywhere on the open set $V$ it is $C^{1,1}$, see \cite{CanSin}. 
Therefore, if we set
$$X(t,x)=\partial_pH\left(x,\partial_x\hat u(t,x)\right),$$
we conclude that the vector  field $\bar X$ on $V$ defined by
 $$\bar X(t,x)=\left(1,\partial_pH\left(x,\partial_x\hat u(t,x)\right)\right)$$
  is locally Lipschitz.
Therefore, we can find a solution $\Gamma:[-\epsilon,\epsilon]\to V$ of $\dot\Gamma(t)=\bar X(t,\Gamma(t))$ with $\Gamma(0)=(\bar t,\bar x)$. 
Given the form of the vector field $\bar X$, we have $\Gamma(s)=(\bar t+s, \bar\gamma(s))$ with $\bar\gamma$ a curve in $M$ such that
$\bar\gamma(0)=\bar x$. 
It is well known that $\gamma:[\bar t-\epsilon,t+\epsilon]\to M$, defined by $\gamma(t)=\bar\gamma(t-\bar t)$,
is $\hat u$-calibrated with $\gamma(\bar t)=\bar x$, see \cite[Proposition 3.4, Page 487]{FaTo}. Proposition \ref{CritCut} now implies that $(\bar t,\bar x)$
is not in  $\operatorname{Cut}_{t_0}({\hat u})$. Hence any open subset of  $]0,t_0[\times M$ intersecting
 $\operatorname{Cut}_{t_0}({\hat u})$ contains points in $\Sigma_{t_0}(\hat u)$.
 \end{proof}
At this point, it is convenient to introduce the cut time function $\tau:]0,t_0[\times M\to]0,t_0]$ for $\hat u$.

\begin{defn}[Cut Time Function] For $(t,x)\in ]0,t_0[\times M$, we define
$\tau(t,x)$ as the supremum of the $t'\in  [t,t_0[$ such that there exists
a $\hat u$-calibrating curve $\gamma:[0,t']\to M$, with $\gamma(t)=x$. The function $\tau:]0,t_0[\times M\to]0,t_0]$ for $\hat u$ is the cut time of $\hat u$.
\end{defn}
We give a characterization of the cut-time function.
\begin{prop}\label{CaracTau} Suppose $(t,x)\in ]0,t_0[\times M$. Choose a $\hat u$-calibrated curve $\gamma:[0,t]\to M$
with  $\gamma(t)=x$. Since $\gamma$ is a minimizer, it can be extended to an extremal $\gamma:[0,+\infty[\to M$
of $L$. We have
$$\tau(t,x)=\sup\{t'\in ]0,t_0[\mid \gamma|[0,t']\text{ is $\hat u$-calibrated}\}\geq t.$$
\end{prop}
\begin{proof} Set $S=\sup\{t'\in ]0,t_0[\mid \gamma|[0,t']\text{ is $\hat u$-calibrated}\}$.
Since $\gamma:[0,t]\to M$ is $\hat u$-calibrated, we indeed have $S\geq t$. Moreover, if $ \gamma|[0,t']$
is $\hat u$-calibrated, with $t'>t$, then$ \gamma|[t,t']$ is also $\hat u$-calibrated.  Therefore, we get $S\leq \tau(t,x)$.

On the other hand if  if $\delta:[t,s]\to M$ is $\hat u$-calibrated, with $\delta(t)=x$, 
then the concatenation $\delta\*(\gamma|[0,t])$ of $\gamma|[0,t]$and $\delta$ is also $d_C$-calibrated, therefore
$\delta\*(\gamma|[0,t])$ is a geodesic. Since it coincides, with $\gamma$ on $0,\delta$, we must have
$\delta\*(\gamma|[0,t)]=\gamma$ on $[0,s]$. In particular, we get that $\gamma|[0,s]$
 is $\hat u$-calibrated. Hence, wet get $\tilde\tau(x)\leq S$.
\end{proof}
\begin{prop}\label{PropTau}
The properties of the cut time function $\tau$ are:
\begin{enumerate}[\rm (i)]
 \item $\tau (t,x)\in [t,t_0]$;
  \item $\tau(t,x)=t$ if and only if $(t,x)\in\operatorname{Cut}_{t_0}(\hat u)$;
  \item $\tau(t,x)=t_0$ if and only if $(t,x)\in{\cal I}_{t_0}(\hat u)$;
  \item the function $\tau$ is upper semi-continuous.
\end{enumerate}
\end{prop} 
\begin{proof} Property (i) is obvious. Property (ii) follows from Proposition \ref{CritCut}.

Property (iii) follows from Proposition \ref{CaracTau} and the fact that calibrated curves are
minimizers (hence extremals).

For part (iv) Assume that $(t_n,x_n)\to (t,x)$ with $\tau(t_n,x_n)\to t'$. We must show that $t'\leq \tau(t,x)$.
By contradiction, assume $t'>\tau(t,x)\geq t$. Fix $t''$ such that $t'>t''>\tau(t,x)\geq t$. 
For $n$ large, we have $\tau(t_n,x_n)> t''>t_n$.
Therefore for such an $n$ we can find a $\hat u$-calibrated curve $\gamma_n:[0,t'']\to M$, with $\gamma_n(t_n)=x_n$.
Extracting, if necessary, we can obtain a curve $\gamma:[0,t'']\to M$ which is a $C^1$ limit of the minimizers
$\gamma_n$. This curve $\gamma$ is also $\hat u$-calibrated and satisfies $\gamma(t)=x$, this a contradiction since
$t''>\tau(t,x)$.
\end{proof}

\section{Local contractibility}
Theorem \ref{LocContr} is a consequence of the following more general one.
\begin{thm}\label{LocContrGen} Let $H:T^*\/M\to \R$ be a Tonelli Hamiltonian. Assume that the function 
$u: M\to [-\infty,+\infty]$ is such that its negative Lax-Oleinik evolution $\hat u$ is finite at every point of $]0,t_0[ \times M$. 
Then the sets $\Sigma_{t_0}(\hat u)$ and  $\operatorname{Cut}_{t_0}({\hat u})$ 
are locally contractible. In particular, they are locally path connected.
\end{thm}
At this point it is useful to introduce the concept of $U$-adapted homotopy.
\begin{defn}[$U$-adapted homotopy]\label{Adapted} Suppose $U:[0,t_0]\times M\to \R$ is a viscosity solution 
of the evolution Hamilton-Jacobi equation \eqref{HJE0}. A continuous homotopy $F:S\times [0,\delta]\to M$,
with $\delta>0$ and $S\subset ]0,t_0[\times M$,
is said to be $U$-adapted if it satisfies
 \begin{enumerate}[(1)]
\item for all $(t,x)\in  S$, we have $t+\delta<t_0$;
\item for all $(t,x)\in  S$, we have $F\left[(t,x),0\right]=x$;
\item if $\left(t+s,F\left[(t,x),s\right]\right)\notin \Sigma_{t_0} (U)$, for some 
 $(t,x)\in S$ and some $s\in ]0,\delta]$, then the curve 
 $\sigma\mapsto F\left[(x, t),\sigma-t\right],\sigma  \in[t,t+s]$, is $U$-calibrated;
\item If  $\gamma:[t,t+s]\to M$ is $U$-calibrated, with $(t,\gamma(t))\in S$,
then for every $\sigma\leq t+\min (s,\delta)$, we have
$F\left[(t,\gamma(t)),\sigma-t\right]=\gamma(\sigma)$.
\end{enumerate}
\end{defn}
\begin{nota}\label{barF} For such a $U$-adapted homotopy $F:S\times [0,\delta]\to M$, with $S\subset ]0,t_0[\times M$,
we define $\bar F: S\times [0,\delta]\to ]0,t_0[\times M$ by
$$\bar F[(t,x),s]=\left(t+s,F\left[(t,x),s\right]\right).$$
\end{nota}
The nice feature of $U$-adapted homotopies is the stability by restriction and composition as given in the following lemma, whose proof is immediate.
\begin{lem}\label{comp} Suppose $U:[0,t_0]\times M\to \R$ is a viscosity solution 
of the evolution Hamilton-Jacobi equation \eqref{HJE0}.
If  $F_1: S_1\times[0, \delta_1]\to M$, and $F_2: S_2\times[0, \delta_2]\to M$ 
are two continuous
$U$-adapted homotopies, with $S_1,S_2\subset ]0,t_0[\times M$ such that $\bar F_1\left[(t,x),\delta_1\right]=
\left(t+\delta_1,F_1\left[(t,x),\delta_1\right]\right)\in S_2$, for all $(t,x)\in S_1$,
then the homotopy $F:S_1\times [0,\delta_1+\delta_2]\to M$ defined by
$$F\left[(t,x),s\right]=\begin{cases}
F_1\left[(t,x),s\right],\text{ for $s\in [0,\delta_1]$},\\
F_2\left[\bar F_1\left[(t,x),\delta_1\right], s-\delta_1\right],\text{ for $s\in [\delta_1,\delta_1+\delta_2]$},
\end{cases}
$$
is itself $U$-adapted. 
\end{lem}

\begin{prop}\label{StrongAdaptedProperty} Assume that the function 
$u: M\to [-\infty,+\infty]$ is such that its negative Lax-Oleinik evolution $\hat u$ is finite at every point of $]0,t_0[ \times M$. If $F:S\times [0,\delta]\to M$, with $S\subset ]0,t_0[\times M$, is a $\hat u$-adapted homotopy, then
we have
$$
\bar F[(t,x),s]=\left(t+s,F\left[(t,x),s\right]\right)\in\Sigma_{t_0}(\hat u), \text{ for all $s\in]\tau(t,x)-t,t_0-t[$,}
$$
where $\tau$ is the cut time function for $\hat u$.
\end{prop}
\begin{proof} Assume that $\bar F[(t,x),s]=\left(t+s,F\left[(t,x),s\right]\right)\notin\Sigma_{t_0}(\hat u)$, then, by part
(3) of Definition \ref{Adapted}, the curve 
 $\sigma\mapsto F\left[(x, t),\sigma-t\right],\sigma  \in[t,t+s]$, is $\hat u$-calibrated which implies that
 $t+s\leq \tau(t,x)$.
 \end{proof}
We will deduce Theorem \ref{LocContrGen} from the lemma below, whose proof will be
postponed to section \S\ref{SecHomLoc}.
\begin{lem}\label{HomLoc} Let $H:T^*\/M\to \R$ be a Tonelli Hamiltonian. Assume that the function 
$u: M\to [-\infty,+\infty]$ is such that its negative Lax-Oleinik evolution $\hat u$ is finite at every point of $]0,t_0[ \times M$. 

Then, for every compact subset $C\subset ]0,t_0[ \times M$, we can find $\delta>0$ and
a $\hat u$-adapted homotopy $F:C\times [0,\delta]\to M$.

Moreover, for every $(t,x)\in C$ and every $s\in]0,\delta]$, we have
\begin{equation}\begin{aligned}
\hat u[t+s, F((t,x),s)]-h_s[x,F((t,x),s)]&=\max_{z\in M}\hat u(t+s, z)-h_s(x,z)\\
&\geq \hat u(t+s, x)-h_s(x,x).
\end{aligned}
\label{CaracF}
\end{equation}
\end{lem}
\begin{proof}[Proof of Theorem \ref{LocContrGen}]
Fix $(\bar t,\bar x)\in \operatorname{Cut}_{t_0}({\hat u})$, and a neighborhood of $(\bar t,\bar x)$ of the form
$[a,b]\times K$, with  $[a,b]\subset ]0,t_0[$ a compact interval, and $K\subset M$ a compact subset. 
Hence $a<\bar t<b$ and $\bar x\in \INT{K}$, where $\INT{K}$ is the interior of $K$. 

We now remark that, to prove the  theorem, it suffices to find a neighborhood $[\bar t-\epsilon,\bar t+\epsilon]\times V$
of  $(\bar t,\bar x)$ contained in
$[a,b]\times K$ and a homotopy
$H:[\bar t-\epsilon,\bar t+\epsilon]\times V\times [0,1]\to [a,b]\times K$ such that

\begin{enumerate}
\item[(C1)] $ H[(t,x),0]= (t,x)$, for all $(t,x)\in [\bar t-\epsilon,\bar t+\epsilon]\times V$;
\item[(C2)]  $H[(t,x),s] \in \Sigma_{t_0}(\hat u)$ for all $(t,x)\in\operatorname{Cut}_{t_0}({\hat u})\cap 
[\bar t-\epsilon,\bar t+\epsilon]\times V$,  and all $s>0$;
\item[(C3)]  $H[(t,x),1] \in \Sigma_{t_0}(\hat u)$,
for all $(t,x)\in [\bar t-\epsilon,\bar t+\epsilon]\times V$.
\end{enumerate}
In fact, properties (C1) and (C2) show that the inclusion 
$\operatorname{Cut}_{t_0}({\hat u})\cap [\bar t-\epsilon,\bar t+\epsilon]\times V  \subset \operatorname{Cut}_{t_0}({\hat u})
\cap [a,b]\times K$ (resp.\ $\Sigma_{t_0}(\hat u)\cap [\bar t-\epsilon,\bar t+\epsilon]\times V  
\subset \Sigma_{t_0}(\hat u)\cap [a,b]\times K $) is homotopic to $H(\cdot,1)$ as maps with values in 
$\operatorname{Cut}_{t_0}({\hat u})\cap [a,b]\times K$ (resp.\  $\Sigma_{t_0}(\hat u)\cap [a,b]\times K $).
We now observe that, cutting down the neighborhood $V$ of $\bar x$ on the manifold $M$, we can assume
that $V$ is contractible in itself. Hence, so is $[\bar t-\epsilon,\bar t+\epsilon]\times V$. Therefore, by (C3),
we obtain that $H(\cdot,1) $  on $[\bar t-\epsilon,\bar t+\epsilon]\times V$ is homotopic to a constant as maps
with values in  $ \Sigma_{t_0}(\hat u)\cap [a,b]\times K$. This clearly finishes the proof of local contractibility for
both  $ \Sigma_{t_0}(\hat u)$ and $\operatorname{Cut}_{t_0}({\hat u})$.

It remains to construct $H$. We first use Lemma \ref{HomLoc} to find a $\hat u$-adapted homotopy $F:[a,b]\times K\times [0,\delta]\to M$
for some $\delta>0$. Since $F\left[(\bar t,\bar x),0\right]=\bar x\in \INT{K}$, we can find 
$\delta_1,\epsilon>0$ and a neighborhood $V\subset \INT{K}$ of $x$ such that 
$[\bar t-\epsilon,\bar t+\epsilon+\delta_1]\subset ]a,b[$, and 
$F\left([\bar t-\epsilon,\bar t+\epsilon]\times V\times [0,\delta_1]\right)\subset \INT{K}$.
Since the point $(\bar t,\bar x)$ is in $\operatorname{Cut}_{t_0}({\hat u})$, we  have $\tau(\bar t ,\bar x)=\bar t$, where
$\tau$ is the cut time function of $\hat u$. Using that $\tau$ is upper semi-continuous, we conclude that
$\tau(t,x)<t+\delta_1$ in a neighborhood of $(\bar t,\bar x)$ in $]0,t_0[ \times M$. Therefore, cutting down
$\epsilon$ and $V$ if necessary, we can assume that $\tau(t,x)<t+\delta_1$,
for $(t,x)\in[\bar t-\epsilon,\bar t+\epsilon]\times V$.

By the upper semi-continuity of $\tau$, using the bound $\tau(t,x)-t<\min\delta_1$ on $[\bar t-\epsilon,\bar t+\epsilon]\times V$ and Baire's interpolation theorem, see \cite[Proposition 7.21]{Brown-Pearcy} 
or \cite[Section VIII.4.3]{Dugundji}, we can find a \emph{continuous} function
$\alpha:[\bar t-\epsilon,\bar t+\epsilon]\times V\to ]0, \delta_1[$ such that
$$0\leq \tau(t,x)-t<\alpha(t,x)<\delta_1, \text{ for all $(t,x)\in [\bar t-\epsilon,\bar t+\epsilon]\times V$.}$$
This implies that $\bar F\left[(t,x),s\alpha(t,x)\right]=\left( t+s\alpha(t,x), F\left[(t,x),s\alpha(t,x)\right]\right)$ is well defined
and is in $[a,b]\times K$,
for all $(t,x)\in [\bar t-\epsilon,\bar t+\epsilon]\times V$ and all $s\in [0,1]$.

Therefore we can now define the continuous homotopy 
$H :[\bar t-\epsilon,\bar t+\epsilon]\times V\times[0,1]\to  [a,b]\times K$ by
\begin{equation}\label{barF01}
H[(t,x),s]=\bar F\left[(t,x),s\alpha(t,x)\right]=\left( t+s\alpha(t,x), F\left[(t,x),s\alpha(t,x)\right]\right).
\end{equation}

Obviously, the homotopy $H$ satisfies the required  property (C1) given above. Since 
$\tau(t,x)<t+\alpha(t,x)$, Proposition \eqref{StrongAdaptedProperty}, applied to the $\hat u$-adapted homotopy $F$,
implies $H[(t,x),1]=\bar F\left[(t,x),\alpha(t,x)\right]\in 
 \Sigma_{t_0}(\hat u)$, for $(t,x)\in [\bar t-\epsilon,\bar t+\epsilon]\times V$. This proves (C3). To prove (C2), we remark that $\tau(t,x)=t$, for $(t,x) \in\operatorname{Cut}_{t_0}({\hat u})$, which, again by  Proposition \eqref{StrongAdaptedProperty},@
 implies $H[(t,x),s]=\bar F\left[(t,x),s\alpha(t,x)\right] \in  \Sigma_{t_0}(\hat u)$, for every $s>0$.
\end{proof}
We now prove Theorem \ref{LocContrDitFunc}.
\begin{proof}[Proof of Theorem \ref{LocContrDitFunc}] If $C$ is a closed subset of the complete Riemannian
manifold $M$, from Example \ref{Exe2}, the negative Lax-Oleinik evolution $\hat \chi_C$ is given, for $t>0$, by
$$\hat \chi_C(x)=\frac{d_C(x)^2}{2t}.$$
The partial derivative $\partial_t\hat \chi_C$ is given by
$$\partial_t\hat \chi_C(t,x)=-\frac{d_C(x)^2}{2t^2}.$$
Hence, it  is defined and continuous everywhere. This implies that 
\begin{equation}\label{EqSing1}
\Sigma(\hat \chi_C)=]0,+\infty[\times\Sigma(d^2_C),
\end{equation}
where $\Sigma(d^2_C)$, as usual, is the set of points in $M$ where $d^2_C$ is not differentiable.

From Theorem \ref{LocContrGen}, we obtain that $\Sigma(\hat \chi_C)=]0,+\infty[\times\Sigma(d^2_C)$ is locally contractible, 
which implies that $\Sigma(d^2_C)$ is also locally contractible. 

We now observe that $d^2_C$ is differentiable at every point $c\in C$, since 
$0\leq d_C^2(x)\leq d^2(c,x)$. Therefore,
from $d_C>0$ on $M\setminus C$, we get 
\begin{equation}\label{EqSing2}
\Sigma(d^2_C)=\Sigma^*(d_C).
\end{equation} 
This finishes the proof Theorem\ref{LocContrDitFunc}.
\end{proof}
\section{Proof of Lemma \ref{HomLoc}}\label{SecHomLoc}
Since $C$ is a compact subset of $]0,t_0[\times M$, it is contained in a subset of the form $[a,b]\times K$,
where $K$ is compact subset of  $M$ and $[a,b]\subset]0,t_0[$.
We then fix a positive number $\eta>0$ such that $[a-\eta,b+\eta]\subset ]0,t_0[$.

By Theorem \ref{PropHatu}, the function $\hat u$ is continuous and even locally semiconcave on
$]0,t_0[ \times M$. This implies that the family of functions $(T^-_{t}u)_{t\in [a-\eta,b+\eta]}$
is equi-continous an equi-semiconcave
on the compact neighborhood $\bar V_1(K)=\{y\in M\mid d_K(y)\leq 1\}$ of $K$.

We define the two continuous functions $u_0,u_1:M\to \R$ by 
\begin{equation}\label{Defu0u1}
u_0(x)=\inf_{t\in[a-\eta,b+\eta]}\hat u(t,x)\text{ and } u_1(x)=\sup_{t\in[a-\eta,b+\eta]}\hat u(t,x).
\end{equation}

We need to use the positive Lax-Oleinik semi-group $T^+_t,t\geq 0$, whose definition we now recall.

If $u:M\to [-\infty,+\infty]$ is a function and $t>0$, the function $T^+_tu:M\to [-\infty,+\infty]$ is defined by
\begin{equation}\label{DefT+}
T^+_tu(x)=\sup_{y\in M}u(y)-h_t(x,y).
\end{equation}
We also set $T^+_0u=u$.

\begin{rem}\label{CaracT+} As is well-known, the semi-group $\check{T}^-_t$ defined by $\check{T}^-_t(u)=-T^+_t(-u)$ is in fact the 
 negative Lax-Oleinik semi-group associated to the Lagrangian $\check{L}:TM\to\R$ defined by $\check{L}(x,v)=L(x,-v)$.
\end{rem}
If $u$ is not identically $-\infty$, then  $T^+_su>-\infty$ everywhere for $s>0$.
Therefore $T^+_su$ is finite everywhere, for $s>0$, if $T^+_su<+\infty$ everywhere and $u$ is not identically $-\infty$.

For $u:M\to [-\infty,+\infty]$, it is convenient to define $\check u:[0,+\infty[\times M\to [-\infty,+\infty]$ by
$$\check u(t,x)=T^+_tu(x).$$
We will call $\check u$ the \emph{positive} Lax-Oleinik evolution.
Using the Remark \ref{CaracT+} above, we obtain from Theorem \ref{PropHatu} the following analogous theorem
\begin{thm}\label{PropChecku} Assume $H:T^*\/M\to \R$ is a Tonelli Hamiltonian. Let $u:M\to [-\infty,+\infty]$ be a function. If its positive Lax-Oleinik evolution $\check u$  is finite-valued on $]0,t_0[\times M$, then it is continuous
and even locally semiconvex on $]0,t_0[\times M$.
\end{thm}

\begin{lem}\label{CoincerT+} For every $s,t\geq0$, we have $T^+_sT^-_{s+t}u\leq T^-_{t}u$.
Therefore $T^+_sT^+_sT^-_{s+t}u$ is locally semi-convex, for every $t\in [a-\eta, b]$, and $s\in ]0,\eta]$.
\end{lem}
\begin{proof}
By the semi-group property of the negative Lax-Oleinik semi-group $T^-_t, t\geq0$, for every $t,s\geq 0$, we have 
$T^-_{s+t}u(x)=\inf_{y\in M}T^-_{t}u(y)+h_s(y,x)$. Therefore
$T^-_{s+t}u(x)\leq T^-_{t}u(y)+h_s(y,x)$,
or equivalently
$T^-_{s+t}u(x)-h_s(y,x)\leq T^-_{t}u(y)$,
which implies $T^+_sT^-_{s+t}u(x)=\sup_{y\in M}T^-_{s+t}u(y)-h_s(x,y)\leq T^-_{t}u(x)$.

For the last part, note that $t+s\in [a-\eta, b+\eta]$, therefore $T^-_{s+t}u$ is finite valued. Moreover,
we have $T^+_sT^-_{s+t}u\leq T^-_{t}u$ which is also finite valued.
Hence, Theorem \ref{PropChecku} implies that $T^+_sT^-_{s+t}u$ is locally semiconvex.
\end{proof}
With $u_0$ and $u_1$ given by \eqref{Defu0u1}, we now define ${\mathcal F}^+_{u_0,u_1, \eta}$ by
$${\mathcal F}^+_{u_0,u_1, \eta}=\{v:M\to \R\mid \text{ $v$ is continuous, $u_0\leq $, and $T^+_{\eta}v\leq u_1$}\}.$$
 By  the very definition \eqref{Defu0u1} of $u_0$, for $t\in [a-\eta,b+\eta]$, we have 
$T^-_{t}u\geq u_0$. Moreover, using Lemma \ref{CoincerT+} and the definition  \eqref{Defu0u1} of $u_1$, 
for $t\in [a,b+\eta]$, we get $ T^+_\eta T^-_{t}u\leq T^-_{t-\eta}u\leq u_1$. This implies 
\begin{equation}\label{T-tuDansF}
T^-_{t}u\in {\mathcal F}^+_{u_0,u_1, \eta},\text{ for all $t\in [a,b+\eta]$.}
\end{equation}

Define $A=\min_{\bar V_2(K)} u_0>-\infty$ and $B=\sup_{\bar V_2(K)}u_1<+\infty$, where as usual 
$\bar V_R(K)=\{y\in M\mid d(y,K)\leq R\}$, for $R\geq 0$. Note that $\bar V_{R'}(\bar V_R(K))\subset
\bar V_{R+R'}(K)$, for $R,R'\geq 0$, Therefore, we have $\bar V_{r}(\bar V_1(K))\subset \bar V_2(K)$, for
$0<r\leq 1$. This implies that
$${\mathcal F}^+_{u_0,u_1, \eta}\subset {\mathcal F}^+_{\bar V_1(K),A,B,\eta,r},$$
for every $0<r\leq 1$, where  as is given by \eqref{DefFABTR} in Appendix \ref{Appendice}.
$${\mathcal F}^+_{\bar V_1(K),A,B,\eta,r}=\{u:M\to[-\infty,+\infty]\mid \text{$A\leq u, T^+_\eta u\leq B$ on 
$\bar V_{r}(\bar V_1(K))$}\},$$
from the definition given by \eqref{DefFABTR} in Appendix \ref{Appendice}.
Therefore, combining with \eqref{T-tuDansF}, for every $0<r\leq 1$, we get
$$T^-_{t}u\in {\mathcal F}^+_{\bar V_1(K),A,B,\eta,r},\text{ for all $t\in [a,b+\eta]$.}$$
From Proposition \ref{PropAppendix}  in Appendix \ref{Appendice}, we obtain the following proposition.
\begin{prop}\label{PropTile5} For  any $0<r\leq 1$, we can find $s(r)>0$, with $s(r)<\eta$ such that
for every $s\in ]0,s(r)]$, every $t\in [a,b]$ and every $x\in \bar V_1(K)$, we have
$$T^+_sT^-_{t+s}u(x)=\sup_\gamma T^-_{t+s}u(\gamma(s))-\int_0^s L(\gamma(\sigma),\dot\gamma(\sigma))\,d\sigma,$$
where the $\sup$ is taken over all  absolutely continuous curve $\gamma:[0,s]\to M$, with $\gamma(0)=x$ 
and $\ell_g(\gamma)\leq r$.

Moreover, we can find a minimizer $\gamma:[0,s]\to M$, with $\gamma(0)=x$, such that
\begin{align*}
T^+_sT^-_{t+s}u(x)&=T^-_{t+s}u(\gamma(s))-h_s(x,\gamma(s))\\
&=T^-_{t+s}u(\gamma(s))-\int_0^s L(\gamma(\sigma),\dot\gamma(\sigma))\,d\sigma.
\end{align*}
Any  absolutely continuous curve
 $\gamma:[0,s]\to M$, 
with $\gamma(0)=x$ and
$$T^+_sT^-_{t+s}u(x)=
T^-_{t+s}u(\gamma(s))-\int_0^s L(\gamma(\sigma),\dot\gamma(\sigma))\,d\sigma.
$$
is a minimizer with $\ell_g(\gamma)\leq r$.

In particular, for $0<s<s(r), t\in[a,b], x\in\bar V_1(K)$, we also have
\begin{equation}\label{T+T-Atr} 
T^+_sT^-_{t+s}u(x)=\sup_{y\in\bar B(x,r)}T^-_{t+s}u(x)u(y)-h_s(x,y).
\end{equation}
Moreover, if $0<s\leq s(r)$, there exists $y\in\bar B(x,r)$ with
\begin{equation}\label{T+T-Atr22} 
T^+_sT^-_{t+s}u(x)=T^-_{t+s}u(y)-h_s(x,y).
\end{equation}
and any $y\in M$ satisfying \eqref{T+T-Atr22} is in $\bar B(x,r)$.
\end{prop}
\begin{cor}\label{CorTile5}  For $0<r\leq 1$,the maps $(s,t,x)\mapsto T^+_sT^-_{t+s}u(x)$ is continuous on
$]0,s(r)]\times [a,b]\times \bar V_1(K)$, where $s(r)>0$ is the number obtained in Proposition 
\ref{PropTile5} above.
\end{cor}
\begin{proof} For $(s,t,x)\in ]0,s(r)]\times [a,b]\times \bar V_1(K)$, we get  from \eqref{T+T-Atr}  in Proposition \ref{PropTile5} that we have
\begin{equation}\label{T+sT-Compact}
T^+_sT^-_{s+t}u(x)=\sup_{y\in \bar V_{r+1}(K)}T^-_{s+t}u(y)-h_s(x,y),
\end{equation}
Since $(s,t,x,y)\mapsto T^-_{s+t}u(y)-h_s(x,y)$ is continuous on $(s,t,x)\in ]0,\eta]\times [a,b]\times \bar V_1(K)
\times \bar V_2(K)$ and $\bar V_{r+1}(K)$ is compact, the continuity on $]0,s(r)]\times [a,b]\times \bar V_1(K)$ of  the map $(s,t,x)\mapsto T^+_sT^-_{s+t}u(x)$ follows.
\end{proof}
The main point toward the existence of a homotopy is the following Lemma.
\begin{lem}\label{Claim1} There exists $\eta'\in ]0,\eta]$, 
such that $T^+_sT^-_{s+t}u$ is $C^1$ (and even $C^{1,1}$)
on $\INT{V}_{1/2}(K)=\{y\in \mid d(y,K)<1/2\}$, for $t\in [a,b]$, and $s\in ]0,\eta']$. 
\end{lem}
As we know already from Lemma \ref{CoincerT+} that $T^+_sT^-_{s+t}u$ is locally semiconvex,
 it suffices to prove that $T^+_sT^-_{s+t}u$ is locally semi-concave on a neighborhood of $\bar V_{1/2}(K)$.

To prove 
Lemma \ref{Claim1}, we use the fact the family 
of functions 
$(T^-_{t}u)_{t\in [a-\eta,b+\eta]}$ 
is equicontinuous and equi-semiconcave on the compact neighborhood  of $\bar V_1(K)$ of $K$. 

If $M$ is compact, we could take $K=M$,
and the $C^1$ property above follows from \cite{Bernard}. 

For the noncompact case, we will use 
\cite[Appendix B]{FatFigRif}, which adapts some of the results of \cite{Bernard} to the noncompact setting. 

To be able to use \cite[Appendix B]{FatFigRif}, it is useful to introduce, for $r>0$
 a modified family $T^{+,r}_t$ of positive Lax-Oleinik operators defined by
$$T^{+,r}_tu(x)=\sup\{ u(y)-h_t(x,y)\mid d(y,x)\leq r\},$$
for $u:M\to [-\infty,+\infty]$.

Using a covering of a compact set by a finite number of domains of charts of the manifold $M$, it is not difficult
to see that the following lemma is a consequence of the positive Lax-Oleinik version of \cite[Lemma B.7 ]{FatFigRif}.
\begin{lem}\label{LemSemiCon} Let $K_1$ and $K_2$ be two compact 
subsets of $M$ with $K_1\subset\INT{K}_2$.

Assume that ${\mathcal F}^+$ is a family of functions from $M$ to $\R$, which is
equi-continuous and equi-semiconcave on $K_2$. 
Then there exits $r_0>0$, and $\delta_0>0$, such that the family 
$\{T^{+,r_0}_tw\mid w\in {\mathcal F}^+, t\in[0,\delta_0]\}$
is  equi-semiconcave on $K_1$.
\end{lem}
\begin{proof}[End of proof of Lemma \ref{Claim1}] We apply Lemma \ref{LemSemiCon} above with 
the family of functions ${\mathcal F}^+=(T^-_tu)_{t\in [a,b+\eta]}$ and the compact sets 
$K_2=\bar V_1(K), K_1=\bar V_{1/2}(K)$. Note that by part \eqref{T+T-Atr}
of Proposition \ref{PropTile5}, we have  
$$T^+_sT^-_{s+t}u(x)=T^{+,r}_sT^-_{s+t}u(x),\text{ for $t\in [a,b],s\in ]0, s(r)]$ and 
$x\in \bar V_1(K)$.}$$ 
Therefore by Lemma \ref{LemSemiCon}, we get that, for 
 $t\in [a,b]$ and $s\leq \eta'=\min(\delta_0,s(r_0))$, the function $T^+_sT^-_{s+t}u$ is semiconcave 
 on $\bar V_{1/2}(K)$.  We  conclude that, for $s\in ]0,\eta'],t\in [a,b]$, the function
 $T^+_sT^-_{s+t}u$ is indeed $C^{1,1}$ on $\INT V_{1/2}(K)$. 
 This finishes the proof of Lemma \ref{Claim1}.
 \end{proof}

As we will show below, the construction of the homotopy in  Lemma \ref{HomLoc} will follow easily from the uniqueness given in the next Lemma.

\begin{lem}\label{Claim2} Fix $\eta'$ given by Lemma \ref{Claim1}.
Assume $(t,x)\in [a,b]\times K$, and $s\in ]0,\eta'[$. Then there exists a unique $y\in M$ such that
\begin{equation}\label{UniquenessYST} T^+_sT^-_{s+t}u(x)= T^-_{s+t}u(y)-h_s(x,y).
\end{equation}
Moreover, if $s\leq \min(\eta',s(r))$, where $s(r)$ is given by Proposition \ref{PropTile5}, we have 
$d(y,x)\leq r$, for the $y$ given by  \eqref{UniquenessYST}

For $(t,x)\in [a,b]\times K,s\in ]0,\eta'[$, and $y$ given by \eqref{UniquenessYST}, the minimizer $\gamma:[0,s]\to M$ joining $x$ to $y$ is also unique, and satisfies
$$d_x(T^+_sT^-_{s+t}u)= \partial_v L(x,\dot{\gamma}(0))\text{ and } \partial_v L(y,\dot{\gamma}(s))\in D^+(T^-_{s+t}u)(y).$$
If $s\leq \min(\eta',s(r))$,, this minimizer $\gamma$  has length $\ell_g(\gamma)\leq r$.
\end{lem}
 To prove this Lemma we need the $T^+_t$ version of Proposition \ref{DiffProp}. 
Again this version follows from Remark \ref{CaracT+}. 
\begin{lem}\label{DiffT+} Suppose $w:M\to\R$ is a continuous function such that $T_s^+w$ is finite, for some $s>0$.
Assume that $\gamma:[0,s]\to M$ is a minimizer such that
$$T_s^+w(\gamma(0))=w(\gamma(s))-h_s(\gamma(0),\gamma(s)),$$
then we have 
\begin{equation}\label{eq:relation_superdifferentials}
\partial_v L(\gamma(0),\dot{\gamma}(0))\in D^-T_s^-w(\gamma(0)) \text{ and } \partial_v L(\gamma(s),\dot{\gamma}(s))\in D^+w(\gamma(s)). 
\end{equation}
In particular, if $T_s^+w$ is differentiable at $\gamma(0)$, then $\gamma(s)$ is the unique $y\in M$, with $T_s^+w(\gamma(0))=w(y)-h_s(\gamma(0),y)$. Moreover, the curve $\gamma$ is the unique minimizer 
$\delta:[0,s]\to M$, with $\delta(0)=\gamma(0)$ and $\delta(s)=\gamma(s)$.
\end{lem}
\begin{proof}[Proof of Lemma \ref{Claim2}] Since $\eta'\leq s(r_0 )$, with $r_0\leq 1$, by \eqref{T+T-Atr22}  in Proposition 
\ref{PropTile5}, we know that there exists at least one $y\in M$ such that
\eqref{UniquenessYST} holds. If we call $\gamma:[0,s]\to M$ the minimizer such that $\gamma(0)=x$ and $\gamma(s)=y$, we have
$$T^+_sT^-_{s+t}u(\gamma(0))= T^-_{s+t}u(\gamma(s))-h(\gamma(0),\gamma(s)).$$
Since $s\leq\eta'$ by Lemma \ref{Claim1}, the function
$T^+_sT^-_{s+t}u$ is differentiable at $x=\gamma(t)$. We can apply Lemma \ref{DiffT+} to deduce that this minimizer 
$\gamma$ is unique, which implies also the uniqueness of $y$. 
\end{proof}

By Lemma \ref{Claim2}, we can  define the function $F:([a,b]\times K)\times ]0,\eta']\to M$ by taking
as $F[(t,x),s]$ the unique $y\in M$ such that
\begin{equation}\label{DefAdapF}T^+_sT^-_{s+t}u(x)= T^-_{s+t}u(y)-h_s(x,y).
\end{equation}
We extend this function $F$  to $([a,b]\times K)\times [0,\eta[$ by $F[(t,x),0]=x$. Therefore $F$ satisfies
part (2) of the Definition \ref{Adapted} of $\hat u$-adapted.
\begin{claim} The homotopy $F:([a,b]\times K)\times [0,\eta']\to M$, defined above, is continuous and $\hat u$-adapted.
\end{claim}
\begin{proof} We first show that $F$ is continuous on $([a,b]\times K)\times ]0,\eta']$. For this, we remark that that $F$
takes values in the compact set $\bar V_2(K)$. Therefore, to show that $F$ is continuous on 
$([a,b]\times K)\times ]0,\eta']$, it suffices to show that its graph is closed in 
$([a,b]\times K)\times ]0,\eta']\times \bar V_2(K)$.
By the definition of $F$, its graph $\Graph(F)$ is given by
\begin{gather*}
\Graph(F)=\{\left((t,x),s,y\right)\in ([a,b]\times K)\times ]0,\eta']\times \bar V_R(K)\mid\\
 T^+_sT^-_{s+t}u(x)= T^-_{s+t}u(y)-h_s(x,y)\}.
\end{gather*}
We now note that we proved in Corollary \ref{CorTile5} that $\left((t,x),s,y\right )\mapsto T^+_sT^-_{s+t}u(x)$ is continuous.
Since tehe functions $[(t,x),s,y]\mapsto T^-_{s+t}u(y)$ and $[(t,x),s,y]\mapsto h_s(x,y)$ are also continuous
for $s>0$, we conclude that $\Graph(F)$ is closed.

To show that $F$ is continuous at a point in $([a,b]\times K)\times \{0\}$, since $y=F[(t,x),s]$,
by Proposition \ref{PropTile5}, we note that $d(x,F[(t,x),s])\leq r$, for $0<s\leq s(r)$.

We now check condition (3) of Definition \ref{Adapted} for $F$. Fix $(t,x)\in [a,b]\times K$.
Assume that $\hat u$ is differentiable at $\left(t+s,F[(t,x),s]\right)$, for some $s>0$.
Choose  a minimizer $\gamma:[t,t+s]\to M$,
with $\gamma(t)=x$ and $\gamma(t+s)=F[(t,x),s]$. Then, by Lemma \ref{Claim2}, we have 
 $$\partial_v L(\gamma(t+s),\dot{\gamma}(t+s))\in D^+(T^-_{s+t}u)(\gamma(t+s)).$$
Since we are assuming that $\hat u(\sigma,z)= T^-_{\sigma}u(z)$ is differentiable
at the point $(s+t,\gamma(t+s))$, by Proposition \ref{DiffProp}, we obtain 
$$d_{\gamma(t+s)}(T^-_{s+t}u)=\partial_v L(\gamma(t+s),\dot{\gamma}(t+s)).$$
We now observe that, by Proposition \ref{DiffProp2}, the backward $\hat u$-characteristic 
$\bar\gamma :[0,t+s]\to M$ ending at  $\gamma(t+s)$ satisfies
$$d_{\bar\gamma(t+s)}(T^-_{s+t}u)=\partial_v L(\bar\gamma(t+s),\dot{\bar\gamma}(t+s)).$$
Since $\gamma(t+s)=\bar\gamma(t+s)$, we conclude that the two extremals $\gamma$ and $\bar\gamma$
have the same position and speed at time $t+s$, therefore $\gamma=\bar \gamma|[t,t+s]$. In particular, the curve
$\gamma$ is $\hat u$-calibrated. To finish the proof that $F$ satisfies part (3) of Definition \ref{Adapted}, and to show,
at the same  time, that $F$ satisfies part (4) of Definition \ref{Adapted}, it suffices to show the following fact:

\noindent\underline{Fact:} If $\gamma:[t,t+s]\to M$ is $\hat u$-calibrated then
$$T^+_sT^-_{s+t}u(\gamma(t))=T^-_{t}u(\gamma(t))=T^-_{s+t}u(\gamma(t+s))-h_s(\gamma(t),\gamma(t+s)).$$
Note that the equality $T^-_{t}u(\gamma(t))=T^-_{s+t}u(\gamma(t+s))-h_s(\gamma(t),\gamma(t+s))$ is equivalent
to $T^-_{s+t}u(\gamma(t+s))=T^-_{t}u(\gamma(t))+h_s(\gamma(t),\gamma(t+s))$, which is a consequence 
of the $\hat u$-calibration of $\gamma$. Therefore, it remains to show that $T^+_sT^-_{s+t}u(\gamma(t))=T^-_{t}u(\gamma(t))$. Note that Lemma \ref{CoincerT+} states that $T^+_sT^-_{s+t}u\leq T^-_{t}u$. 
On the other hand, by definition of $T^+_sT^-_{s+t}u$, we have $T^+_sT^-_{s+t}u(\gamma(t))\geq T^-_{s+t}u(\gamma(t+s))-h_s(\gamma(t),\gamma(t+s))=T^-_{t}u(\gamma(t))$.
 \end{proof}
To finish the proof of Lemma \ref{HomLoc}, it remains to observe that  \eqref{CaracF}
follows from the definition of $F$.
\section{Constructions of Global Homotopy Equivalences}
The goal of this section is to establish some results that will help to construct homotopies of the type mentioned
in the proof of Theorem \ref{HomGlob}.
\begin{prop}\label{ReducGlobHom} Assume $u:M\to[-\infty,+\infty]$ is such that $\hat u$ is finite on $]0,+\infty[\times M$.
If, for every compact subset $C\subset ]0,+\infty[\times M$, we can find a $\hat u$-adapted homotopy
$F:C\times [0,1]\times M$, then, for every $T\in ]0,+\infty]$, the inclusions  
$$\Sigma_{T}(\hat u)\subset\operatorname{Cut}_{T}({\hat u})\subset ]0,+\infty[\times M\setminus
{\cal I}_{T}(\hat u)$$
are all homotopy equivalences.
\end{prop}
We reduce, by several lemmas, the proof of Proposition \ref{ReducGlobHom}.
\begin{lem}\label{HomLocGlobal} Under the hypothesis of Proposition \ref{ReducGlobHom}, for every compact subset $C\subset ]0,+\infty[\times M$,
there exists a $\hat u$-adapted homotopy $F:C\times [0,+\infty[\to M$.
\end{lem}
\begin{proof} We show, by induction on the integer $n\geq 1$, how to extend the  $\hat u$-adapted homotopy $F:C\times [0,1]\to M$ to a
 $\hat u$-adapted homotopy 
$F:C\times [0,n]\to M$. Assume $F:C\times [0,n]\to M$ is constructed for some $n\geq 1$. As introduced in Notation  \ref{barF},
we define $\bar F:C\times [0,n]\to ]0,+\infty[ \times M$ by
$$\bar F[(t,x),s]=\left(t+s, F[(t,x),s]\right).$$
Since $\bar F$ is continuous, the subset $C_n=\bar F[C\times \{n\}]$ is a compact subset of $ ]0,+\infty[ \times M$.
Therefore, by the hypothesis of Proposition \ref{ReducGlobHom} applied to $C_n$ instead of $C$, 
we can find a $\hat u$-adapted homotopy $F_n:C_n\times [0,1]\to M$. By Lemma \ref{comp}, if we extend $F$ to 
$C\times [n,n+1]$, by $F[(t,x),s]=F_n\left[\bar F[(t,x),n],s-n\right]$, it will be $\hat u$-adapted on $C\times [0,n+1]$.
\end{proof}
\begin{rem} In fact, if, instead of the hypothesis of Proposition \ref{ReducGlobHom}, we assume that there exists 
$\delta>0$ such that for every compact subset $C\subset ]0,+\infty[\times M$, we can find a
 $\hat u$-adapted homotopy $F:C\times [0,\delta]\times M$, then the conclusion of
  Lemma \ref{HomLocGlobal} remains valid.
\end{rem}
\begin{lem}\label{LastLemma} Under the hypothesis of Proposition \ref{ReducGlobHom}, for every compact subset   $C$ 
of $]0,T[\times M\setminus {\cal I}_T(\hat u)$, where $T>0$,  we can find 
a continuous homotopy 
$$G:\left(]0,T[\times M\setminus {\cal I}_T(\hat u)\right)\times [0,1]
\to ]0,T[\times M\setminus {\cal I}_T(\hat u)$$ 
such that :
\begin{align}
\label{555} G[(t,x),1]&\in \Sigma_T(\hat u), \text{for all $(t,x)\in C$,}\\
\label{556} G[(t,x),0)&=(t,x),\text{ for all $(t,x)\in ]0,T[\times M\setminus {\cal I}_T(\hat u) $,}
\end{align} 
and 
\begin{align}
\label{665}  G(\Sigma_T(\hat u)\times [0,1])&\subset \Sigma_T(\hat u),\\ 
\label{666} G(\operatorname{Cut}_T(\hat u)\times [0,1])&\subset \operatorname{Cut}_T(\hat u).
\end{align} 
\end{lem}
\begin{proof} We choose $C'$ a compact neighborhood of $C$ in $]0,T[\times M\setminus {\cal I}_T(\hat u)$. We apply
Lemma \ref{HomLocGlobal} for $C'$ (instead of $C$) to obtain the homotopy
$F:C'\times [0,+\infty[\to  M$. As introduced in Notation  \ref{barF},
we define $\bar F:C'\times [0,+\infty[\to ]0,+\infty[ \times M$ by
$$\bar F[(t,x),s]=\left(t+s, F[(t,x),s]\right).$$
 We first observe that the image of $\bar F$ avoids ${\cal I}_T(\hat u)$. In fact, if 
 $(t+s,y)= \bar F[(t,x),s]=\left(t+s, F[(t,x),s]\right)$ is in ${\cal I}_T(\hat u)$, then $0<t+s<T$, and there exists
 a $\hat u$-calibrated curve $\gamma:[0,T[\to M$ such that $\gamma(t+s)=y$. By Proposition \ref{DiffProp2}, this implies that $\hat u$ is differentiable
 at $(t+s,y)$ and any $\hat u$-calibrated curve $\delta:[a, t+s]\to M$ ending at $y$ must coincide with $\gamma|[a,t+s]$.
 Using again that $\hat u$ is differentiable at $(t+s,y)= \bar F[(t,x),s]=\left(t+s, F[(t,x),s]\right)$, 
 and that $F$ is $\hat u$-adapted, we obtain that the curve 
 $\sigma\to F[(t,x),\sigma-t], \sigma\in [t,t+s]$, which ends at $y$,  is $\hat u$-calibrated. Therefore
 $F[(t,x),\sigma-t]=\gamma(\sigma)$, and $\bar F[(t,x),\sigma-t]\in {\cal I}_T(\hat u)$, for $\sigma \in [t,t+s]$. In particular,
 for $\sigma=t$,
 this would imply $(t,x)\in {\cal I}_T(\hat u)$, which is impossible, since $(t,x)$ is in $C'$ which is disjoint from 
 ${\cal I}_T(\hat u)$. Thus we obtained 
 \begin{equation}\label{ITU}
 F:C'\times ]0,+\infty[\to M\setminus{\cal I}_T(\hat u).  
\end{equation}
Since $F$ is $\hat u$-adapted, we must have
\begin{align}
\label{775}\bar F\left((C'\cap \Sigma_T(\hat u))\times [0,+\infty[\right)&\subset  \Sigma_T(\hat u)\\
\label{776}\bar F\left((C'\cap \operatorname{Cut}_T(\hat u))\times [0,+\infty[\right)&\subset  
\operatorname{Cut}_T(\hat u).
\end{align}
Since $C'$ is a compact subset of $]0,T[\times M\setminus{\cal I}_T(\hat u)$, and the upper semi-continuous  
cut time function $\tau$ is $<T$ everywhere in $]0,T[\times M\setminus{\cal I}_T(\hat u)$, we can choose a finite $T_0<T$ such that the cut time function $\tau$ is $< T_0$ on  $C'$. Since $C'$ is a compact neighborhood of $C$
in $]0,T[\times M\setminus{\cal I}_T(\hat u)$, we can find a continuous function $\alpha: C'\to [0,1]$ such that $\alpha$
is identically equal to $1$ on $C$, and $\alpha$ is identically $0$ on $\partial C'$, the boundary of $C'$ in 
$]0,T[\times M\setminus{\cal I}_T(\hat u)$. We define $G:C'\times [0,1]\to ]0,+\infty [\times M$ by
$$G[(t,x),s]=\bar F[(t,x), s\alpha(t,x)(T_0-t)].$$
In particular, we get 
$$G[(t,x),s]=\left(t+s\alpha(t,x)(T_0-t), F[(t,x), s\alpha(t,x)(T_0-t)]\right).$$
Since $t+s\alpha(t,x)(T_0-t)\leq T_0$, the image $G(C'\times[0,1])$ of $G$ is, in fact, 
contained in $]0,T_0 ]\times M\setminus{\cal I}_T(\hat u)\subset $. Taken together with \eqref{ITU}, it implies
$$G(C'\times[0,1])\subset ]0,T[\times M\setminus {\cal I}_T(\hat u).$$
Properties \eqref{556},\eqref{665} and\eqref{666} of $G$ on $C'\times [0,1]$ follow from properties 
\eqref{775} and \eqref{776} of $\bar F$.
Since for $(t,x)\in C$, we have $\alpha(t,x)=1$, we get
$$G[(t,x),1]=\left(t+s\alpha(t,x)(T_0-t), F[(t,x), (T_0-t)]\right).$$
Since $F$ is $\hat u$-adapted  and $\tau<T_0$ on $C'\supset C$, we obtain property \eqref{555}.

Since $G[(t,x),s]=(t,x)$, for $(t,x)\in \partial C'$, we can extend $G$ continuously to
$(]0,T_0 [\times M\setminus{\cal I}_T(\hat u))\times [0,1]$ by $G[(t,x),s]=(t,x)$, for $(t,x)\notin C'$.

It is not difficult to check that this extension $G$ still has the required properties \eqref{555} to \eqref{666}.
\end{proof}
\begin{proof}[Proposition \ref{ReducGlobHom}]
We can find a sequence of compact subsets
$C_n, n\geq 1$ of $]0,T [\times M\setminus{\cal I}_T(\hat u)$, such that $C_n\subset \INT{C}_{n+1}$, and
 $]0,T [\times M\setminus{\cal I}_T(\hat u)=\cup_{n\geq 0}C_n$.
We construct a homotopy $H:\left(]0,T [\times M\setminus{\cal I}_T(\hat u)\right)\times [0,+\infty[\to
]0,T [\times M\setminus{\cal I}_T(\hat u)$ such that 
$ H(\Sigma_T(\hat u)\times [0,+\infty[)\subset \Sigma_T(\hat u), H(\operatorname{Cut}_T(\hat u)\times [0,+\infty[)\subset \operatorname{Cut}_T(\hat u), H((t,x,0)=(t,x)$, for all $(t,x)\in ]0,T [\times M\setminus{\cal I}_T(\hat u)$,
and
$$ H(C_n\times [n+1,+\infty[) \subset \Sigma_T(\hat u),\text{ for all $n\geq 0$.}$$
We will construct $H$ on $]0,T [\times M\setminus{\cal I}_T(\hat u)\times [0,n]$ by induction on $n\geq 1$.

We start by applying Lemma \ref{LastLemma} to the compact set $C_1$ to obtain the homotopy
$$G_1:\left(]0,T [\times M\setminus{\cal I}_T(\hat u)\right)\times [0,1]\to ]0,T [\times M\setminus{\cal I}_T(\hat u),$$
with 
$$ G_1[(t,x),0]=(t,x),\text{ for all $(t,x)\in \left(]0,T [\times M\setminus{\cal I}_T(\hat u)\right)$,}$$
and 
\begin{align*}
G_1(C_1\times \{1\}&\subset \Sigma_T(\hat u),\\ 
 G_1(\Sigma_T(\hat u)\times [0,1])&\subset \Sigma_T(\hat u),\\ 
 G_1(\operatorname{Cut}_T(\hat u)\times [0,1])&\subset \operatorname{Cut}_T(\hat u).
\end{align*}

We then set $H[(t,x),s]=G_1[(t,x),s]$, for $(t,x)\in 0,T [\times M\setminus{\cal I}_T(\hat u),s\in [0,1]$.
Assuming that $H$ has been constructed on $\left(]0,T [\times M\setminus{\cal I}_T(\hat u)\right)\times [0,n]$, we apply
Lemma \ref{LastLemma} to the compact set $H(C_n\times\{n\})$ to obtain the
homotopy 
$$G_n:\left(]0,T [\times M\setminus{\cal I}_T(\hat u)\right)\times [0,1]\to ]0,T [\times M\setminus{\cal I}_T(\hat u),$$ 
with
$$ G_n[(t,x),0]=(t,x),\text{ for all $(t,x)\in \left(]0,T [\times M\setminus{\cal I}_T(\hat u)\right)$,}$$
and 
\begin{align*}
G_n[H(C_n\times\{n\})\times \{1\}]&\subset \Sigma_T(\hat u),\\ 
 G_1(\Sigma_T(\hat u)\times [0,1])&\subset \Sigma_T(\hat u),\\ 
 G_1(\operatorname{Cut}_T(\hat u)\times [0,1])&\subset \operatorname{Cut}_T(\hat u).
\end{align*} 
We then define $H$ on
$\left(]0,T [\times M\setminus{\cal I}_T(\hat u)\right)\times [n,n+1]$ by 
$$H[(t,x),s]= G_n(H[(t,x),n], t-n).$$
It is not difficult to check that $G$ satisfies the required properties \eqref{555}--\eqref{666}.

Since $C_n\subset \INT{C}_{n+1}$, we can define a continuous function 
$\alpha:]0,T [\times M\setminus{\cal I}_T(\hat u)\to\R$ such that $\alpha=n+2$ on $\partial C_n$,
$\alpha(C_0)\subset [1,2]$, and $\alpha(C_n\setminus \INT{C}_{n-1})\subset [n+1,n+2]$.
Therefore $H[(t,x),\alpha(t,x)]\in \Sigma_T(\hat u)$, for all $(t,x)\in ]0,T [\times M\setminus{\cal I}_T(\hat u)$.

We then define the homotopy $\tilde H:M\setminus {\cal I}(u)\times [0,1]\to M\setminus {\cal I}(u)$ by
$$\tilde H[(x,t),s]=H[(t,x),s\alpha(t,x)].$$
It is not difficult to check that $\tilde H[(x,t),0]=(x,t), 
\tilde H[(x,t),1]\in \Sigma_T(\hat u)$, for all $(t,x)\in ]0,T [\times M\setminus{\cal I}_T(\hat u)$, 
$\tilde H(\Sigma_T(\hat u)\times [0,1])\subset \Sigma_T(\hat u)$, 
and $\tilde H(\operatorname{Cut}_T(\hat u)\times [0,1])\subset \operatorname{Cut}_T(\hat u)$.
This finishes the proof of Proposition \ref{ReducGlobHom}.
\end{proof}
Here is a useful criterion that allows to show that the hypothesis of Proposition \ref{ReducGlobHom} holds.
\begin{lem}\label{ReducGlobHombis} Let $u:M\to[-\infty,+\infty]$ be such that $\hat u$ is finite on $]0,+\infty[\times M$.
Assume for every $t_0>0$, we can find a constant $\kappa_{t_0}$, such that, for every compact subset 
$C\subset [t_0,+\infty[\times M$, we can find $\delta>0$ and $\hat u$-adapted homotopy
$F:C\times [0,\delta]\to M$, such that
$$d(F[(t,x),s],x)\leq \kappa_{t_0}s,\text{ for all $(t,x)\in C$ and all $s\in  [0,\delta]$}.$$
Then for every compact subset $C\subset ]0,+\infty[\times M$, we can find a $\hat u$-adapted homotopy
$F:C\times [0,1]\times \to M$.
\end{lem}
\begin{proof}
Assume $C$ is a compact subset of $]0,+\infty[\times M$.We first find $a,b\in ]0,+\infty[$, with $a<b$, and a compact subset $K\subset M$ such that $C\subset [a,b]\times K$.

We now use the hypothesis of the Lemma, applied to the compact set $[a,b+1]\times \bar V_{\kappa_a}(K)
\subset [a,+\infty[\times M$, to find
a $\hat u$-adapted homotopy
$$F:\left([a,b+1]\times \bar V_{\kappa_a}(K)\right)\times[0,\delta]\to M,$$
such that, for all $(t,x)\in [a,b+1]\times \bar V_{\kappa_a}(K)$ and all $s\in[0,\delta]$, we have
\begin{equation}\label{BoundDistHomotopyBis}
d\left(F[(t,x),s], x\right)\leq \kappa_as.
\end{equation}

If $\delta\geq 1$ we have finished. If $\delta<1$, choose $n_0\leq 2$ such that $(n_0-1)\times \delta<1\leq n_0\delta$.
Note that $n_0\delta<1+\delta\leq 2$.
By induction on $n=1,\dots, n_0$, we will construct an extension of $F$ on $C\times [0,\delta]$ to 
a $\hat u$-adapted homotopy $F:C\times [0,n\delta]\to M$ which also satisfies
\begin{equation}\label{BoundDistHomotopyBisInd}
d\left(F[(t,x),s], x\right)\leq \kappa_as, 
\text{ for all $(t,x)\in C$ and all $s\in[0,n\delta]$.}
\end{equation}
For $n=1$, we just take the restriction of $F$ to $C\times [0,\delta]$. Assuming $F:C\times [0,n\delta]\to M$ constructed
with $n<n_0$, we note that for every $(t,x)\in C\subset [a,b]\times K$ and every $s\in [0,n\delta]\subset [0,1]$, we have
$t+s\in [a,b+1]$ and $d(F[(t,x),s],x)\leq \kappa_as\leq \kappa_a$. Hence $\left(t+n\delta, F[(t,x),n\delta]\right)\in [a,b+1]\times \bar V_{\kappa_a}(K)$. Therefore, for $s\in [n\delta,(n+1)\delta]$, we can set 
$$F[(t,x),s]=F\left[(t+n\delta, F[(t,x),n\delta]), s-n\delta\right].$$
By Lemma \ref{comp}, this extension of $F$ is still $\hat u$-adapted.
It remains to check \eqref{BoundDistHomotopyBisInd} for $s\in [n\delta,(n+1)\delta]$. Noting that
$F[(t,x),s]=F\left[(t+n\delta, F[(t,x),n\delta]),s\!-\!n\delta\right]$,  from \eqref{BoundDistHomotopyBis},
we obtain $d(F[(t,x),s],F[(t,x),n\delta])\leq \kappa_a(s-n\delta)$,  hence 
\begin{align*}
d(F[(t,x),s],x)&\leq d(F[(t,x),s], F[(t,x),n\delta])+d(F[(t,x),n\delta],x)\\
&\leq \kappa_a(s-n\delta)+\kappa_an\delta=\kappa_as.
\end{align*}
Since $n_0\delta>1$, this finishes the proof of the Lemma.
\end{proof}
\section{Functions Lipschitz in the large}
To state the generalization we have in mind, we recall the definition of Lipschitz in the large for a function, 
see \cite[Definition A.5]{Zav} or \cite{FatNC}.

\begin{defn} Let $X$ be a metric space whose distance is denoted by $d$.
A function $u:X\to \R$ is said to be Lipschitz in the large if there exists a constant $K<+\infty$ such that
$$\lvert u(y)-u(x)\rvert\leq K+Kd(x,y), \text{ for every $x,y\in X$.}$$
When the inequality above is satisfied, we will say that $u$ is Lipschitz in the large with constant $K$.
\end{defn}

Note that we do not assume in the definition above that $u$ is continuous. 

Obviously, when $X$ is compact $u:X\to \R$ is  Lipschitz in the large if and only if $u$ is bounded.

As is shown in \cite[Proposition 10.3]{FatNC},
the function $u:X\to \R$ is Lipschitz in the large if and only if there exits a (globally) Lipschitz function $\varphi:X\to \R$
such that 
$$\lVert u-\varphi\rVert_\infty=\sup_{x\in X}\lvert u(x)-\varphi(x)\rvert<+\infty.$$
In particular, a Lipschitz in the large function $u:M\to\R$ is bounded from below by a Lipschitz function and therefore
$\hat u$ is finite everywhere on $[0,+\infty[\times M$.

As we will see below the  next theorem generalizes Theorem \ref{HomGlob} stated in the introduction.
\begin{thm}\label{HomGlobGen} Assume $u:M\to \R$ is a Lipschitz in the large function. For every $T\in ]0,+\infty]$
the inclusions $\Sigma_T(\hat u)\subset \operatorname{Cut}_T(\hat u)\subset ]0,T[ \times M\setminus {\cal I}_T(\hat u)$
are homotopy equivalences.
\end{thm}
It is not difficult to see that, for a function $u:M\to \R$ Lipschitz in the large with constant $K$, its lower semi-continuous regularization $u_-$ is itself Lipschitz in the large with constant $K$. Therefore by Proposition \ref{hatu-}, without loss
of generality we can prove Theorem \ref{HomGlobGen} adding the assumption that $u$ is lower semi-continuous.
The main step is to show the following global version of Lemma \ref{HomLoc}.

We now recall the following result, whose proof is standard and can be found in \cite[Theorem 10.4]{FatNC}.
\begin{prop}\label{LipHatu} If $u$ is Lipschitz in the large, then
for every $t_0>0$, the restriction of its negative Lax-Oleinik evolution $\hat u$ to $[t_0,+\infty[\times M$ is (globally) Lipschitz. \end{prop}
\begin{cor}\label{BoundDistHomotopy} If $u$ is Lipschitz in the large, then
for every $t_0>0$, we can find a constant $\kappa_{t_0}$ such that for every $t\geq t_0, s>0, x,y\in M$ with
$\hat u(t+s, y)-h_s(x,y)\geq \hat u(t+s, x)-h_s(x,x)$, 
we have
$d(x,y)\leq \kappa_{t_0}s$.
\end{cor}
\begin{proof} By Proposition \ref{LipHatu} above we can find a global Lipschitz constant $\lambda<+\infty$ for
$\hat u$ on $[t_0,+\infty[\times M$. 
The inequality $\hat u(t+s, y)-h_s(x,y)\geq \hat u(t+s, x)-h_s(x,x)$ yields
$$h_s(x,y)]\leq \hat u(t+s, y)-\hat u(t+s, x)+h_s(x,x).$$
Since $h_s(x,x)\leq sA(0)$  (where
$A(0)=\sup_{x\in M}L(x,0)$ was defined before, see\eqref{TonLag3}), we obtain
\begin{align*}
h_s(x,y)&\leq \hat u(s+t,y)-\hat u(s+t,x)+h_s(x,x)\\
&\leq\lambda d(x,y)+A(0)s.
\end{align*}
But by the uniform superlinearity of $L$, see \eqref{TonLag2bis}, we know that $h_s(x,y)\geq (\lambda+1)d(x,y)-
C(\lambda+1)s$. Hence, we have $(\lambda+1)d(x,y)-C(\lambda+1)s\leq\lambda d(x,y)+A(0)s$, from which the inequality
$d(x,y)\leq [C(\lambda+1)+A(0)]s$ follows.
\end{proof}
\begin{proof}[Proof of Theorem \ref{HomGlobGen}] Suppose $C\subset ]0,+\infty[\times M$ is compact.
We show that there exists a $\hat u$-adapted homotopy $F:C\times [0,+\infty[\to M$. 
This will imply Theorem \ref{HomGlobGen} by Proposition \ref{ReducGlobHom}.

Choose $t_0>0$ such that $C\subset ]t_0,+\infty[\times M$. By Lemma \ref{HomLoc}, we can find a $\hat u$- adapted
homotopy $F_0:C\times [0,\delta]\to M$ with
$$
\hat u[t+s, F_0((t,x),s)]-h_s[x,F_0((t,x),s)]\geq \hat u(t+s, x)-h_s(x,x).
$$
By Corollary \ref{BoundDistHomotopy}, we obtain  $d(F_0((t,x),s),x)\leq \kappa_{t_0}$. 
The existence of  $\hat u$-adapted homotopy $F:C\times [0,+\infty[\to M$ now follows from 
Lemma \ref{ReducGlobHom} and Lemma \ref{HomLocGlobal}.
\end{proof}

\begin{proof}[Proof of Theorem \ref{HomGlob}] Under the hypothesis
of Theorem \ref{HomGlob}, by Theorem \ref{UniquenessUC}, we know that $U=\hat u$, with $u:M\to\R$ defined by $u(x)=U(0,x)$. Note that the function $u$ itself is uniformly continuous.

Any uniformly continuous function $u:M\to\R$ is Lipschitz in the large, since it is a uniform limit of Lipschitz functions, see for example
\cite[Lemma 7.9]{FatNC}. Therefore, Theorem \ref{HomGlob} is a consequence of the  more general 
Theorem \ref{HomGlobGen}.
\end{proof}
\begin{cor}\label{AlmostLipLarCase}
Assume that the function $u:M\to [-\infty,+\infty]$, with $\hat u$ finite everywhere on $]0,+\infty[$, 
is such that $T_{t}^-u=\hat u(t,\cdot)$ is Lipschitz
in the large, for every $t>0$. Then, for every $T>0$, the inclusions 
 $\Sigma_T(\hat u)\subset \operatorname{Cut}_T(\hat u)\subset ]0,T[ \times M\setminus {\cal I}_T(\hat u)$
are homotopy equivalences.
\end{cor} 
\begin{proof} Again, like in the proof of Theorem \ref{HomGlobGen}, it suffices to show that, for every
compact subset $C\subset ]0,+\infty[\times M$, there exists a $\hat u$-adapted homotopy $F:C\times [0,+\infty[\to M$. 

By compactness of $C$, there exists  $t_0>0$ such that $C\subset ]t_0,+\infty[\times M$. 
Set $u_{t_0}=T_{t_0}^-u=\hat u(t_0,\cdot)$. 
We have $\hat u_{t_0}(t,x)=\hat u(t_0+t,x)$. Setting $C_{t_0}=\{(t-t_0,x)\mid (t,x)\in C\}\subset
]0,+\infty[$, since $u_{t_0}$ is Lipschitz in the large, by the beginning of the proof
of Theorem \ref{HomGlobGen}, we know that we can find a $\hat u_{t_0}$-adapted homotopy 
$F_0:C_{t_0}\times [0,+\infty[\to M$. It is not difficult to see that $F:C\times [0,+\infty[\to M$, defined 
by $F(t,x)=F_0(t-t_0,x)$, is a $\hat u$-adapted homotopy.
\end{proof}
\begin{cor}\label{CompactCase} Suppose that the manifold $M$ is compact. 
If the function $u:M\to [-\infty,+\infty]$ is such that
$\hat u$ is finite everywhere on $]0,+\infty[\times M$, then, for every $T>0$, the inclusions 
 $\Sigma_T(\hat u)\subset \operatorname{Cut}_T(\hat u)\subset ]0,T[ \times M\setminus {\cal I}_T(\hat u)$
are homotopy equivalences
\end{cor} 
\begin{proof} By Theorem \ref{PropHatu}, we know that $T^-_tu$ is continuous on $M$, for every $t>0$,
therefore Lipschitz in the large, since $M$ is compact. It suffices now to apply Corollary \ref{AlmostLipLarCase}.
\end{proof}
\begin{lem}\label{CompactComponentDist} Assume that $C\subset M$ is a closed subset of the complete Riemannian manifold $(M,g)$.
If $U$ is a relatively compact component of $M\setminus C$, then  $]0,+\infty[\times U$ is disjoint from ${\cal I}_{+\infty}(\hat \chi_C)$.

In particular, if $M$ is compact, then for any closed subset $C$
we have ${\cal I}_{+\infty}(\hat \chi_C)=]0,+\infty[\times C$.
\end{lem}
\begin{proof}Assume $(t_0,x_0)$, with $x_0\in U$ and $t_0>0$, is in the Aubry set ${\cal I}_{+\infty}(\hat \chi_C)$. We can find a $\hat \chi_C$-calibrated curve $\gamma:[0,+\infty[\to M$, with $\gamma(t_0)=x_0$. By Lemma \ref{CacarCaldC}, the curve $\gamma$
	is a $g$-geodesic satisfying $\gamma(0)\in C$ and
	$d_C(\gamma(t))=d(\gamma(0),\gamma(t))$, for all $t\geq 0$.
	
	Since $\gamma$ is a $g$-geodesic, its speed
	$\lVert\dot\gamma(s)\rVert_{\gamma(s)}$ is a constant $v\geq 0$ independent of $t\in [0,+\infty[$. Therefore, 
	the equality above is
    \begin{equation}\label{VTT}
    d_C(\gamma(t))=vt.
    \end{equation}
	In particular, we get $d_C(\gamma(t_0))=vt_0$. But $\gamma(t_0)=x_0\notin C$, which implies $d_C(\gamma(t_0))>0$.
	Hence $v> 0$. We then deduce from \eqref{VTT} that $\gamma(t)\notin C$. Since $\gamma(t_0)=x_0\in U$, this implies that $\gamma(t)$ is in the connected component $U$ of $M\setminus C$, for all $t>0$. But the continuous function $d_C$ is bounded on $U$. This contradicts \eqref{VTT} since $v>0$.
\end{proof}
\begin{cor}\label{Compact Dist} Assume that $C\subset M$ is a closed subset of the \emph{compact} Riemannian manifold $M$.
Then the inclusion $\Sigma^*(d_C)\subset M\setminus C$ is a homotopy equivalence.
\end{cor}
\begin{proof} Since $M$ is compact, by the previous 
	Lemma  \ref{CompactComponentDist}
	${\cal I}_{+\infty}(\hat \chi_C)\subset ]0,+\infty[\times C$.
	In fact, we have 
	$${\cal I}_{+\infty}(\hat \chi_C)=]0,+\infty[\times C,$$
	since any constant curve in $C$ is obviously $\hat\chi_C$-calibrated.

By Corollary \ref{CompactCase}, we  obtain that the inclusion 
$\Sigma^*(\hat \chi_C)\subset ]0,+\infty[\times (M\setminus C)$ is a homotopy equivalence. But, by \eqref{EqSing1}, we have
$\Sigma^*(\hat \chi_C)=]0,+\infty[\times\Sigma(d^2_C)$ and by \eqref{EqSing2}, we have
$\Sigma(d^2_C)=\Sigma^*(d_C)$. This finishes the proof of the Corollary.
\end{proof}
\section{More application to complete non-compact  Riemannian manifold}\label{SecDist}
Assume that $C$ is a closed subset of the complete Riemannian manifold $(M,g)$. 
If $M$ is not compact, then neither $\chi_C$  nor $\hat\chi_C(t_0,\cdot)=d_C(\cdot)^2/{2t}$ are necessarily 
Lipschitz in the large. So, we cannot apply Corollary \ref{AlmostLipLarCase} to obtain the global homotopy type of 
$\Sigma^*(d_C)=\Sigma(d^2_C)$. Instead, we will show that
Proposition \ref{ReducGlobHom} directly applies.

For $y\in M$ and $t,s>0$,  we introduce the function
$\varphi_{t,s,y}:M\to \R$ defined by
\begin{equation}
\begin{aligned}
\varphi_{t,s,y}(x)&=\hat\chi_C(t+s,x)-h^g_s(x,y)\\
	&=\frac{d_C(x)^2}{2(t+s)}-\frac{d(y,x)^2}{2s}.
\end{aligned}
\label{DefVarphi}
\end{equation}
\begin{lem}\label{BorneDistPoint}
	Suppose that $x,y\in M$ are such
	$ \varphi_{t,s,y}(x)\geq \varphi_{t,s,y}(y)$, then we have
	\begin{align*}
	d(x,y)&\leq \frac{2s}td_C(y)\\
	d_C(x)&\leq \left(1+\frac{2s}t\right)d_C(y).
	\end{align*}
\end{lem}
\begin{proof}
From the definition \eqref{DefVarphi}, the inequlity $ \varphi_{t,s,y}(x)\geq \varphi_{t,s,y}(y)$ translates to
$$\frac{d_C(x)^2}{2(t+s)}-\frac{d(y,x)^2}{2s}\geq \frac{d_C(y)^2}{2(t+s)}.$$
Using $d_C(x)\leq d_C(y)+d(x,y)$, we obtain
\begin{align*}
\frac{d_C(y)^2}{2(t+s)}&\leq \frac{[d_C(y)+d(x,y)]^2}{2(t+s)}-\frac{d(y,x)^2}{2s}\\
&=\frac{d_C(y)^2}{2(t+s)}+\frac{2d_C(y)d(x,y)}{2(t+s)}
+\left(\frac{1}{2(t+s)}-\frac1{2s}\right)d(x,y)^2\\
&=\frac{d_C(y)^2}{2(t+s)}+\frac{2d_C(y)d(x,y)}{2(t+s)}
-\frac {td(x,y)^2}{2s(t+s)}d(x,y)^2,
	\end{align*}
from which the inequality $d(x,y)\leq 2sd_C(y)/t$  follows.

To prove the last inequality of the lemma, we again use
$d_C(x)\leq d_C(y)+d(x,y)$.
\end{proof}

\begin{lem}\label{BorneDistSuite}
Assume that $ (t_0,x_0),\dots,( t_n, x_n),\dots$ is a  (finite or infinite) sequence in $]0,+\infty[\times M$, with  $s_n=t_{n+1}-t_{n}\geq 0$, and such that
$$ \varphi_{t_{n},s_{n},x_{n}}(x_{n+1})\geq  \varphi_{t_{n},s_{n},x_{n}}(x_{n}).$$
We have
\begin{align*}
d_C(x_n)&\leq e^{2t_n/t_0}d_C(x_0)\\
d(x_n,x_0)&\leq 2e^{3t_n/t_0}d_C(x_0).
\end{align*}
\end{lem}
\begin{proof}
From Lemma \ref{BorneDistPoint}, we get
\begin{align}
\label{IneqSuite1}d(x_{n+1},x_n)&\leq \frac{2s_n}{t_n}d_C(x_n)\\
\label{IneqSuite2} d_C(x_{n+1})&\leq \left(1+\frac{2s_n}{t_n}\right)d_C(x_n).
\end{align}
Therefore, from \eqref{IneqSuite2}, we get
 $$d_C(x_n)\leq  \prod_{i=0}^{n-1}\left(1+\frac{2s_i}{t_i}\right)d_C(x_0).$$
To estimate the quantity $C_n=\prod_{i=0}^{n-1}\left(1+\frac{2s_i}{t_i}\right)$, we take its $\log$
$$\log C_n=\sum_{i=0}^{n-1}\log\left(1+\frac{2s_i}{t_i}\right).$$
Since $\log(1+t)\leq t$, for $t\in ]0,+\infty[$, we obtain
$$\log C_n\leq 	2\sum_{i=0}^{n-1}\frac{s_i}{t_i}.$$
Since $t_i$ is non-decreasing and $\sum_{i=0}^{n-1}s_i=t_n-t_0\leq t_n$, we get  
$$\log C_n\leq 2\frac{\sum_{i=0}^{n-1}s_i}{t_0}\leq\frac{2t_n}{t_0},$$
which proves the first inequality in the lemma. 
By the inequality \eqref{IneqSuite1} and the already established first inequality of the lemma, we have
\begin{align*}
d(x_n,x_0)&\leq \sum_{i=0}^{n-1}d(x_{i+1},x_i)\\
&\leq \sum_{i=0}^{n-1}\frac{2s_i}{t_i}d_C(x_i)\\
&2d_C(x_0)\sum_{i=0}^{n-1}\frac{s_i}{t_i}e^{2t_i/t_0}
\end{align*}
Using that $t_i$ is non-decreasing and the fact that $\sum_{i=0}^{n-1}s_i\leq t_n$, we obtain 
 
\begin{align*}
d(x_n,x_0)&\leq 2d_C(x_0)\sum_{i=0}^{n-1}\frac{s_i}{t_0}e^{2t_n/t_0}\\
&=2d_C(x_0)\frac{\sum_{i=0}^{n-1}s_i}{t_0}e^{2t_n/t_0}.
\end{align*}
Therefore, since $\sum_{i=0}^{n-1}s_i\leq t_n$, we get
$$d(x_n,x_0)\leq =2\frac{t_n}{t_0}e^{2t_n/t_0}d_C(x_0),$$
which finishes the proof of the lemma, since $t\leq\exp(t)$.
\end{proof}
We  now prove that $\chi_C$ satisfies the hypothesis of Proposition
\ref{ReducGlobHom}. 

\begin{prop}\label{ReducGlobDist} Suppose $C$ is a closed subset of the complete Riemannian manifold $(M,g)$.Then
for every compact subset $K\subset ]0,+\infty[\times M$, we can find a $\hat \chi_C$-adapted homotopy
	$F:K\times [0,1]\to M$.
\end{prop}
\begin{proof}
The compact subset $K\subset ]0,+\infty[\times M$ is contained in a set of the form $[a,b] \times  A$, where $A$ is a compact subset of $M$
and $0<a<b<+\infty$. We then define $\kappa$ by
$$\kappa= e^{3(b+1)/A}\max_Ad_C<+\infty.$$

Since $A$ is compact, its neighborhood $\bar V_\kappa(A)=\{x\in M\mid
d_A(x)\leq \kappa \}$ is also compact. By  Lemma \ref{HomLoc}, we can find $\delta>0$ and a $\hat \chi_C$-adapted homotopy
$F:\left([a,b]\times \bar V_\kappa(A)\right)\times [0,\delta]\to M$.

If $\delta\geq 1$, then the restriction of $F $ to $K\times [0,1]$ does the job. If not, cutting down on $\delta$ we assume $\delta=1/n$,
with $n$ an integer $\geq 2$.

We will show that we can extend $F$ by induction on $i=1,\dots, n-1$ to $F:K\times [0,i/n]\to M$ by 
$$F((t,x),s)=F[F(t+i/n,x), s-i/n], \text{ for $s\in[i/n,(i+1)/n]$.}$$
This homotopy is well-defined if we show by induction that,
for $(t,x)\in K\subset [a,b]\times A$, 
the sequence $(x_i,t+i/n)$, with $x_0=x$, defined by induction for
$i= 1,\dots n-1$ as
$$x_i=F(t+i/n,x_{i-1}), 1/n],$$
is such that $x_i\in \bar V_\kappa(A)$, for $i=1,\dots, n-1$.

Since  $F:([a,b]\times \bar V_\kappa(A))\times [0,1/n]\to M$
is defined by  Lemma \ref{HomLoc}, the sequence $x_i$, while it make sense for $i=0,1,\dots,n-1$, satisfies the hypothesis of Lemma \ref{BorneDistSuite} with
$t_i=t+ is\leq t+1$. Therefore 
$$d(x_i,x_0)\leq e^{3t_{i}/t}d_C(x_0)\leq e^{3(b+1)/a}\max_Ad_C=\kappa,$$
since $x_0=x\in A,t\in [a,b]$ and $i\leq n-1$. This implies that
$x_i\in \bar V_\kappa(A)$.
\end{proof}
\begin{proof}[Proof of Theorem \ref{GlobalDistFunc}]
Since we know by Proposition \ref{ReducGlobDist} that $\hat \chi_C$
satisfies the hypothesis of Proposition \ref{ReducGlobHom}, we
obtain that the inclusion
$\Sigma_\infty(\hat \chi_C)\subset ]0,+\infty[\times M \setminus
{\cal I}_\infty(\hat \chi_C)$ is a homotopy equivalence. But, by
\eqref{EqSing1} and \eqref{EqSing2}, we have
$\Sigma_\infty(\hat \chi_C)=]0,+\infty[\times\Sigma(d^2_C)=]0,+\infty[\times\Sigma^*(d_C)$
and, by Proposition \ref{AubryDistProp}, we have 
${\cal I}_\infty(\hat \chi_C)=]0,+\infty[\times (C\cup {\cal A}^*(C))$.
Therefore, the inclusion $]0,+\infty[\times\Sigma^*(d_C)\subset 
]0,+\infty[\times \left(M \setminus(C\cup {\cal A}^*(C)) \right)$
is a homotopy equivalence, which implies that 
$ \Sigma^*(d_C)\subset 
(M \setminus(C\cup {\cal A}^*(C)) $ is itself a homotopy equivalence.
\end{proof}
We now explain the generalization of Theorem \ref{HomTypeNU} to non-compact complete Riemannian manifold.

We first introduce the subset ${\cal AU}(M,g)\subset M\times M$.
\begin{defn} For a complete connected Riemannian manifold $(M,g)$, the subset 
${\cal AU}(M,g)\subset M\times M$ is the set of points $(x,y)\in M\times M$ such that there exists a
\emph{minimizing} $g$-geodesic $\gamma:]-\infty,+\infty[\to M$, with
$\gamma(a)=x,\gamma(b)=y$, for some $a<b\in \R$.
\end{defn}
The next Lemma is left to the reader.
\begin{lem} We have
	$${\cal AU}(M,g)=\Delta_M\cup {\cal A}^*(d_{\delta_M}).$$
\end{lem}

The following generalization of Theorem \ref{HomTypeNU} now follows
from the Lemma above and Theorem \ref{GlobalDistFunc}.
\begin{thm} For every  connected complete Riemannian manifold $M$, the inclusion
	${\cal NU}(M,g)\subset M\times M\setminus  {\cal AU}(M,g)$ is a homotopy equivalence.
\end{thm}
\section{More results on local contractibility}\label{SecLocCon}
In fact, our local contractibility result can be applied for viscosity solution defined only on an open subset, and also for
Hamiltonians which  are not necessarily uniformly superlinear.
More precisely, we have.
\begin{thm} Suppose $H:T^*M\to\R$  is a C\/$^2$ Hamiltonian such that:
\begin{enumerate}[(a)]
\item\label{HSu} (Superlinearity on compact subsets) For every compact subset $C$ and every  real number $K\geq 0$, we have
$$\sup\{ K\lVert p\rVert_x-H(x,p)\mid x\in C, p\in T^*_xM\}<\infty.$$
\item ($C^2$ strict convexity in the fibers) For every $(x,p)\in T^*\/M$, 
the second derivative along the fibers $\partial^2 H /\partial 
p^2(x,p)$ is positive definite.
\end{enumerate}
If  the continuous function $U:O\to \R$ defined on the open subset $O\subset \R\times M$ is a viscosity solution
of 
\begin{equation}\label{HJEf}
\partial_tU+ H(x,\partial_xU)=0,
\end{equation}
then its set of singularities $\Sigma(U)\subset O$ is locally contractible.
\end{thm}
\begin{proof} Since the result is local, we can assume that $O={} ]a,b[\times W$, with $W\subset M$ open with 
$\bar W$ compact. We will first modify $H$ outside  of $T^*W$ to reduce to the case where the Hamiltonian is Tonelli.
Choose a C$^\infty$ function $\varphi:M\to[0,1]$ with compact support such that $\varphi$ is identically $1$ on the compact 
subset $\bar W$. Define $\tilde H:T^*M\to\R$ by
$$\tilde H(x,p)= (1-\varphi(x))\lVert p\rVert^2_x+ \varphi(x)H(x,p).$$
Since the support of $\varphi$ is compact, using the properties of $H$, it is not difficult to check that $\tilde H$
is Tonelli and coincides with $H$ on $T^*W$. This last fact implies that $U: ]a,b[\times W\to \R$ is also a viscosity 
solution of \eqref{HJEf} for $\tilde H$.

Therefore without loss of generality, we can assume that $H$ is Tonelli.

Fix $(t_0,x_0)\in ]a,b[\times W$. Pick $\eta >0 $ 
such that $[t_0-\eta,t_0+\eta]\times \bar B(x_0,2\eta)\subset ]a,b[\times W$.

The viscosity solution $U$ is Lipschitz on the compact set $[t_0-\eta,t_0+\eta]\times \bar B(x_0,2\eta)$. Therefore,
since speeds of $U$-calibrating curves are related to upper differentials of $U$ by the Legendre transform, we can find a
finite constant $K$ such that for every $[\alpha,\beta]\subset [t_0-\eta,t_0+\eta]$ and every curve $\gamma:
[\alpha,\beta]\to \bar B(x_0,2\eta)$, which is $U$-calibrating, we have $\lVert \dot\gamma(s)\rVert_{\gamma(s)}\leq K$.
Choose now $\epsilon>0$ such that $\epsilon <\eta$ and $2K\epsilon<\eta$. Since $U:O\to\R$, as any viscosity
solution, has backward characteristic ending at any given point, from the definition of $K$ and the choice
of $\epsilon$, for every $(t,x)\in ]t_0-\epsilon, t_0+\epsilon]\times\bar B(x_0,\eta)$, we can find a
$U$-calibrated curve $\gamma_{t,x}:[t_0-\epsilon,t]\to \bar B(x_0,2\eta)$. It follows that
$$U(t,x)=\inf_{y\in\bar B(x_0,2\eta)}U(t_0-\epsilon, y)+ h_{t-t_0+\epsilon}(y,x),$$
for all $(t,x)\in ]t_0-\epsilon, t_0+\epsilon]\times\bar B(x_0,\eta)$.
Therefore, if we define $u:M\to \R$ by
$$u(y)=\begin{cases} U(t_0-\epsilon, y),\text{ if $y\in\bar B(x_0,2\eta)$,}\\
+\infty, \text{ if $y\in\bar B(x_0,2\eta)$,}
\end{cases}
$$
we obtain
\begin{equation}\label{ExtLocVis}
U(t,x)=\hat u(t-t_0+\epsilon,x)\text{ for all $(t,x)\in ]t_0-\epsilon, t_0+\epsilon]\times\bar B(x_0,\eta)$.}
\end{equation}
Since $U$ is continuous and $\bar B(x_0,2\eta)$ compact and not empty, the function $\hat u$ is bounded below 
and  not identically $+\infty$. Therefore $\hat u$ is finite everywhere and $\Sigma^*(\hat u)$ is locally contractible.
By \eqref{ExtLocVis}, we have $\Sigma(U)\cap ]t_0-\epsilon, t_0+\epsilon[\times \INT B(x_0,\eta)=
\{(s+t_0-\epsilon,x)\mid (s,x)\in \Sigma^*(\hat u)\cap ]0, 2\epsilon[\times \INT B(x_0,\eta)$. Hence 
$\Sigma(U)\cap ]t_0-\epsilon, t_0+\epsilon[\times\INT B(x_0,\eta)$ is itself locally contractible.
\end{proof}
\begin{appendix}
\section{Some estimates}\label{Appendice}
In this appendix, we suppose that $H:T^*M\to\R$ is a Tonelli Hamiltonian on the complete Riemannian
manifold $(M,g)$ and $L$ is its associated Lagrangian. The Lax-Oleinik operators $T^-_t$ and $T^+_t$
are the ones associated to $L$.

Suppose $K$ is a compact subset of $M$. If $A,B, t,r$ are finite real numbers with $t,r>0$, we set
\begin{equation}\label{DefFABTR}
{\mathcal F}^+_{K,A,B,t,r} =\{u: M\to [-\infty,+\infty]\mid \text{$A\leq u$ and $T^+_tu\leq B$ on  $\bar V_r(K)$}\},
\end{equation}
where as usual $\bar V_r(K)=\{y\in M\mid d(y,K)\leq r\}$ is the closed $r$-neighborhood of $K$ in $M$.

The goal of this appendix is to prove the following proposition.
\begin{prop}\label{PropAppendix} For  any compact subset $K\subset M$, 
and any finite $A,B, t,r$, with $t,r>0$, we can find
$s_0=s(K,A,B, t,r)>0$ such that for any function $u\in {\mathcal F}^+_{K,A,B,t,r}$,
any $x\in K$ and any $0<s\leq s_0=s(K,A,B,t,r)$, every absolutely continuous curve $\gamma:[0,s]\to M$, 
with $\gamma(0)=x$ and
$$u(\gamma(s))-\int_0^s L(\gamma(\sigma),\dot\gamma(\sigma))\,d\sigma+1\geq T^+_su(x).$$
must satisfy $\ell_g(\gamma)\leq r$.
Therefore, for $u\in {\mathcal F}^+_{K,A,B,t,r},x\in M$ and $0<s\leq s_0=s(K,A,B,t,r)$, we have:
$$T^+_su(x)=\sup_\gamma u(\gamma(s))-\int_0^s L(\gamma(\sigma),\dot\gamma(\sigma))\,d\sigma,$$
where the $\sup$ is taken over all  absolutely continuous curve $\gamma:[0,s]\to M$, 
with $\gamma(0)=x$ and $\ell_g(\gamma)\leq r$. In particular, we have
$$T^+_su(x)=\sup_{y\in \bar B(x,r)}u(y)-h_s(x,y).$$

Moreover, if $u$ is \emph{continuous} on $\bar V_r(K)$, the set
$$\{y\in M\mid T^+_su(x)=u(y)-h_s(x,y)\}$$
is non-empty compact and contained in the closed ball $\bar B(x,r) $. In particular, we can find
a \emph{minimizer} $\gamma:[0,s]\to M$, with $\gamma(0)=x$, and
\begin{equation}\label{MinForT+sx}
T^+_su(x)=u(\gamma(s))-\int_0^s L(\gamma(\sigma),\dot\gamma(\sigma))\,d\sigma.
\end{equation}
Any  absolutely continuous curve $\gamma:[0,s]\to M$, 
with $\gamma(0)=x$, satisfying  the equality  \eqref{MinForT+sx} above is a minimizer with $\ell_g(\gamma)\leq r$.

\end{prop}
To prove this proposition, we need a couple of lemmas.
\begin{lem}\label{MajActAndLength1}
For every finite $A$ and $r>0$, we can find a constant $\eta=\eta(A,r)$ such that
for every absolutely continuous curve $\gamma:[a,b]\to M$, with $b-a\leq \eta$ and
$\Act(\gamma)=\int_a^bL(\gamma(s),\dot\gamma(s))\,ds\leq A$, we have $\ell_g(\gamma)\leq r$,
where $\ell_g(\gamma)$ is the Riemannian length of $\gamma$ for the Riemannian metric $g$ on $M$.

Therefore, if $h_t(x,y)\leq A$, with $x,y\in M$ and $0<t\leq \eta=\eta(A,r)$, we have $d(x,y)\leq r$.
\end{lem}
\begin{proof} From \eqref{TonLag2bis}, with $K=2\lvert A\rvert /r$, we have
\begin{equation}\label{TonLagApplique} 
\forall (x,v)\in TM, L(x,v)\geq \frac{2\lvert A\rvert }r \Vert v\Vert_x -C(2\lvert A\rvert /r).
\end{equation}
If $\gamma:[a,b]\to M$ is an  absolutely continuous curve $\gamma:[a,b]\to M$, applying \eqref{TonLagApplique}
at $(x,v)=(\gamma(s),\dot\gamma(s))$ and integrating, we obtain
$$\Act(\gamma)\geq  \frac{2\lvert A\rvert }r \ell_g(\gamma)-C(2\lvert A\rvert /r)(b-a).$$
If $\Act(\gamma)\leq A$, this yields
$$A\geq  \frac{2\lvert A\rvert }r \ell_g(\gamma)-C(2\lvert A\rvert /r)(b-a),$$
or equivalently
$$\ell_g(\gamma)\leq \frac r2 +\frac {rC(2\lvert A\rvert /r)}{2\lvert A\rvert }(b-a).$$
It suffices to set
$$\eta=\eta(A,r)=\frac{\lvert A\rvert } {C(2\lvert A\rvert /r)}.$$
The last part follows from the fact that $h_t(x,y)$ is the action of a minimizer
$\gamma:[0,t]\to M$, with $\gamma(0)=x$ and $\gamma(t)=y$.
\end{proof}

\begin{lem}\label{BoundednessMaximizer}
Suppose $K$ is a compact subset of $M$. If $A,B, t,r$ are finite real numbers with $t,r>0$,
we can find a constant $s_0=s(K,A,B,t,r)>0$, with $s_0\leq t$, such that for every $u\in {\mathcal F}^+_{K,A,B,t,r}$,
every $x\in M$, every $0<s\leq s_0=s(K,A,B,t,r)$, and every absolutely continuous curve $\gamma:[0,s]\to M$, 
with $\gamma(0)=x$ and $\ell_g(\gamma)>r$, we have
$$T^+_su(x)> u(\gamma(s))-\int_0^s L(\gamma(\sigma),\dot\gamma(\sigma))\,d\sigma+1.$$
\end{lem}
\begin{proof}
We first note that for $u\in {\mathcal F}^+_{K,A,B,t,r},
x\in V_r(K)$ and  $0<s\leq t$, we have
$$
T^+_su(x)\geq u(x)-h_s(x,x)\geq A-A(0)s,
$$
where $A(0)=\sup_{x\in M}L(x,0)<+\infty$.
Therefore if we set $\tilde A=A-\lvert A(0)\rvert t$, we have
\begin{equation}\label{MinT+}
T^+_su(x)\geq u(x)-h_s(x,x),
\end{equation}
for $u\in {\mathcal F}^+_{K,A,B,t,r},x\in V_r(K)$ and  $s\leq t$.
On the other hand for $0<s<t$, by the semi-group property, we have
$$T^+_tu(x)\geq T^+_su(x)-h_{t-s}(x,x)\geq T^+_su(x)-A(0)(t-s).$$
Hence for $0<s<t$ and $x\in V_r(K)$, we get
$$
T^+_su(x)\leq B+A(0)(t-s)\geq \bar A.
$$
Setting $\tilde B= B+\lvert A(0)\rvert t$, this yields
\begin{equation}\label{MaxT+}
T^+_su(x)\leq \bar B,
\end{equation}
for $u\in {\mathcal F}^+_{K,A,B,t,r},x\in V_r(K)$ and  $s\leq t$.

To prove the lemma, we must find $s_0=s(K,A,B,t,r)>0$ such that any absolutely continuous curve 
$\gamma:[0,s]\to M$, 
with $0<s\leq t, \gamma(0)=x\in$ and $ K\ell_g(\gamma)>r$, such that
\begin{equation}\label{EquaInter}
u(\gamma(s))-\int_0^s L(\gamma(\sigma),\dot\gamma(\sigma))\,d\sigma+1\geq T^+_su(x)
\end{equation}
must satisfy $s>s_0=s(K,A,B,t,r)$. 

Assume now that 
 $\gamma:[0,s]\to M$ is a curve satisfying the conditions above.
Since $\ell_g(\gamma)>r$, we can find $s'\in ]0,s[$ such that $\ell_g(\gamma[0,s'])=r$.
Using \eqref{EquaInter}, we get
$$u(\gamma(s))-\int_{s'}^s L(\gamma(\sigma),\dot\gamma(\sigma))\,d\sigma
-\int_0^{s'} L(\gamma(\sigma),\dot\gamma(\sigma))\,d\sigma+1\geq T^+_su(x).$$
But $T^+_{s-s'}u(\gamma(s'))\geq u(\gamma(s))-\int_{s'}^s L(\gamma(\sigma),\dot\gamma(\sigma))\,d\sigma$. Hence
$$T^+_{s-s'}u(\gamma(s'))-\int_0^{s'} L(\gamma(\sigma),\dot\gamma(\sigma))\,d\sigma+1\geq T^+_su(x).$$
Since $x\in K$ , from \eqref{MinT+}, we have $T^+_su(x)\geq \bar A$. Moreover, since $\ell_g(\gamma[0,s'])=r$
and $x=\gamma(0)=x\in K$,
we get $\gamma(s')\in \bar V_r(K)$, which, by \eqref{MaxT+}, implies $T^+_{s-s'}u(\gamma(s'))\leq \bar B$.
From the last inequality, we conclude that 
$$\int_0^{s'} L(\gamma(\sigma),\dot\gamma(\sigma))\,d\sigma\leq \bar B-\bar A+1.$$
By Lemma \ref{MajActAndLength1}, using $\ell_g(\gamma[0,s'])=r$,  this last inequality
 implies that $s'>\eta( \bar B-\bar A+1,r/2)$.
Since $s>s'$ to finish the proof of the Lemma, it suffices to take $s_0=s(K,A,B,t,r)=\eta( \bar B-\bar A+1,r/2)$.
\end{proof}
\begin{proof}[Proof of Proposition \ref{PropAppendix}] We consider the $s_0=s(K,A,B,t,r)$ given by the last Lemma
\ref{BoundednessMaximizer}. Given a function $u\in {\mathcal F}^+_{K,A,B,t,r},x\in K$ and  $0<s\leq s_0=s(K,A,B,t,r)$.
By Lemma \ref{BoundednessMaximizer}, we have
\begin{equation}\label{GOGO1}
T^+_su(x)> u(\gamma(s))-\int_0^s L(\gamma(\sigma),\dot\gamma(\sigma))\,d\sigma+1,
\end{equation}
for every absolutely continuous curve $\gamma:[0,s]\to M$, 
with $\gamma(0)=x$ and $\ell_g(\gamma)>r$. Since
\begin{equation}\label{GOGO2}
T^+_su(x)=\sup_\gamma u(\gamma(s))-\int_0^s L(\gamma(\sigma),\dot\gamma(\sigma))\,d\sigma,
\end{equation}
where the $\sup$ is taken over all  absolutely continuous curve $\gamma:[0,s]\to M$, 
with $\gamma(0)=x$, we obtain
\begin{equation}\label{GOGO3}
T^+_su(x)=\sup_\gamma u(\gamma(s))-\int_0^s L(\gamma(\sigma),\dot\gamma(\sigma))\,d\sigma,
\end{equation}
where the $\sup$ is taken over all  absolutely continuous curve $\gamma:[0,s]\to M$, 
with $\gamma(0)=x$ and $\ell_g(\gamma)\geq r$.
Obviously, equality \eqref{GOGO3} and inequality \eqref{GOGO1} imply
\begin{align*}
T^+_su(x)&=\sup_{y\in \bar B(x,r)}u(y)-h_s(x,y)\\
&>\sup_{y\notin \bar B(x,r)}u(y)-h_s(x,y).
\end{align*}
When $u$ is, moreover, continuous on $\bar V_r(K)$, since $\bar B(x,r)\subset \bar V_r(K)$, the  first equality 
above shows that
$\{y\in M\mid T^+_su(x)=u(y)-h_s(x,y)\}$ is not empty,
and the second inequality shows that this set is
contained in  $\bar B(x,r)$. The compactness of $\{y\in M\mid T^+_su(x)=u(y)-h_s(x,y)\}$ follows from 
the continuity of $u$ on $\bar B(x,r)\subset \bar V_r(K)$.

Since, we can find $y\in M$, with $T^+_su(x)=u(y)-h_s(x,y)$, taking a minimizer $\gamma:[0,s]\to M$
with $\gamma(0)=x$ and $\gamma(s)=y$, we obtain
\begin{equation}\label{MinForT+sx4}
T^+_su(x)=u(\gamma(s))-\int_0^s L(\gamma(\sigma),\dot\gamma(\sigma))\,d\sigma.
\end{equation}
Moreover, for any absolutely continuous curve $\gamma:[0,s]\to M$, 
with $\gamma(0)=x$, satisfying \eqref{MinForT+sx4}, we get from \eqref{GOGO1} and \eqref{GOGO2}
that $\gamma$ is a minimizer with length $\ell_g(\gamma)\leq r$.
\end{proof}
Our goal now is to sketch a proof of Theorem \ref{PropHatu}.
For this we will need the $T^-_t$ version of  Proposition \ref{PropAppendix}. Thanks to Remark \ref{CaracT+}, this version follows from Proposition \ref{PropAppendix}.

We first define
\begin{equation}\label{DefFABTR--}
{\mathcal F}^-_{K,A,B,t,r} =\{u: M\to [-\infty,+\infty]\mid \text{ $A\leq T^-_tu$ and $u\leq B$ on  $\bar V_r(K)$}\},
\end{equation}
\begin{prop}\label{PropAppendix-} For  any compact subset $K\subset M$, 
and any finite $A,B, t,r$, with $t,r>0$, we can find
$s_0=s(K,A,B, t,r)>0$ such that for any function $u\in {\mathcal F}^-_{K,A,B,t,r}$,
any $x\in K$ and any $0<s\leq s^-_0=s(K,A,B,t,r)$, every absolutely continuous curve $\gamma:[0,s]\to M$, 
with $\gamma(s)=x$ and
$$T^u(\gamma(0))+\int_0^s L(\gamma(\sigma),\dot\gamma(\sigma))\,d\sigma-1\geq T^-_su(x).$$
must satisfy $\ell_g(\gamma)\leq r$.
Therefore, for $u\in {\mathcal F}^+_{K,A,B,t,r},x\in M$ and $0<s\leq s^-_0=s(K,A,B,t,r)$, we have:
$$T^-_su(x)=\inf_\gamma u(\gamma(0))+\int_0^s L(\gamma(\sigma),\dot\gamma(\sigma))\,d\sigma,$$
where the $\sup$ is taken over all  absolutely continuous curve $\gamma:[0,s]\to M$, 
with $\gamma(s)=x$ and $\ell_g(\gamma)\leq r$. In particular, we have
$$T^-_su(x)=\inf_{y\in \bar B(x,r)}u(y)+h_s(y,x).$$

Moreover, if $u$ is \emph{continuous} on $\bar V_r(K)$, the set
$$\{y\in M\mid T^-_su(x)=u(y)+h_s(x,x)\}$$
is non-empty, compact, and contained in the closed ball $\bar B(x,r) $.
\end{prop}
\begin{proof}[Proof of Theorem \ref{PropHatu}] Fix now $u:M\to [-\infty,+\infty]\to M$ and $t_0>0$ such that $\hat u$ is finite on $]0,t_0[$.
If $K$ is a compact subset and $0<t_1<t_2<t_0$, we will show that $\hat u$ is continuous on 
 and locally semiconcave on $]t_1,t_2[\times \INT K$. This will finish the proof of the continuity 
 and semiconcavity of $\hat u$ on $]0,t_0[\times M$.

Since $\hat u$ is not identically $+\infty$ or $-\infty$, there is a point $x_0\in M$ such that $u(x_0)$ is finite.
We have
$$ T^-_tu(x)\leq u(x_0) +h_t(x_0,x),\text{ for all $t>0$ and $x\in M$.}$$
Fix now $t'_2$ with $t_2<t'_2<t_0$
Since $(t,y,z)\mapsto h_t(y,z)$ is continuous on $]0,+\infty[\times M\times M$, we have
$\sup\{h_t(x_0,x)\mid t\in[t_1,t'_2], x\in \bar V_1(K)<+\infty$. Therefore, we can find a finite constant $B$
such that
\begin{equation}\label{ConstB-}
T^-_tu(x)\leq B,\text{ for all $t\in[t_1,t'_2]$ and $x\in \bar V_1(K)$.}
\end{equation}
Fix now $t''_2$ such that $t'_2<t''_2<t_0$, 
$$ T^-_{t''_2}u(x_0)\leq T^-_tu(x) +h_{t''_2-t}(x,x_0),\text{ for all $0<t<t'_2$ and $x\in M$,}$$
or equivanlently
$$ T^-_{t''_2}u(x_0)-h_{t'_2-t}(x,x_0)\leq T^-_tu(x) ,\text{ for all $0<t<t'_2$ and $x\in M$.}$$
Since, for $t\in [t'_1,t'_2]$, we have  $0<t''_2-t'_2\leq t''_2-t\leq t''_2-t_1$, we conclude as above that
there exists a finite constant $A$ such that
\begin{equation}\label{ConstA-}
A\leq T^-_tu(x)\leq B,\text{ for all $t\in[t_1,t'_2]$ and $x\in \bar V_1(K)$.}
\end{equation}
Setting now $\eta=t'_2-t_2>0$, from \eqref{ConstB-} and \eqref{ConstA-}, we conclude that
$$T^-_tu\in {\mathcal F}^-_{K,A,B,\eta,1},\text{ for all $t\in[t_1,t_2]$.}$$
Therefore, we can apply Proposition \ref{PropAppendix-}, to find $s^-_0$, with $0<s^-_0<\eta$ such that
$$T^-_{s+t}u(x)=T^-_sT_tu(x)=\inf_{y\in \bar B(x,1)}T_tu(y)+h_s(y,x )\text{ for all $t\in[t_1,t_2]$ and all $x\in K$,}$$
which of course implies
\begin{equation}\label{EXPT-}
T^-_{s+t}u(x)=\inf_{y\in \bar V_1(K)}T_tu(y)+h_s(y,x )\text{ for all $t\in[t_1,t_2]$ and all $x\in K$.}
\end{equation}
Fix now $\bar t\in ]t_1,t_2[\times K$ and pick $\delta>0$ such that $3\delta<s^-_0$ and $\bar t-2\delta\geq t_1$.
For $t\in [\bar t-\delta,\bar t+\delta]$, we have $t-(\bar t-2\delta)\in [\delta,3\delta]$. Since $3\delta<s^-_0$, and 
$\bar t-2\delta\geq t_1$, from \eqref{EXPT-}, we obtain
\begin{equation}\label{FEXPT_}
T^-_{t}u(x)=\inf_{y\in \bar V_1(K)}T_{(\bar t-2\delta}u(y)+h_{t-(\bar t-2\delta)}(y,x ),
\text{ for all $t\in[\bar t-\delta,\bar t+\delta]$ and all $x\in K$.}
\end{equation}

Since this map  $(s,x,y)\mapsto h_s(y,x)$ is locally Lipschitz and
locally semi concave on $]0,+\infty[\times M\times M$ and $\bar V_1(K)$ is compact, we obtain that the family of maps
 $(s,x)\mapsto T_{(\bar t-2\delta}u(y)+h_s(y,x), y\in\bar V_1(K)$ is locally uniformly semiconcave on a neighborhood of the compact set $ [\delta,3\delta]\times \times K$. The local semiconcavity (hence the continuity) of $\hat u$ on
 a neighborhood of $[\bar t-\delta,\bar t+\delta]\times K$, follows now, for example, from \eqref{FEXPT_} and
 \cite{FatFig}[Corollary A.14], which states that a pointwise finite $\inf$ of a family of locally uniformly semiconcave
functions is itself locally semiconcave.

We just established that $\hat u$ is continuous and even locally semiconcave on $]0,t_0[\times M$.
Note that since we know that $\hat u$ is finite on $]0,t_0[\times M$, we get
$$\hat u(t,x)=\inf_{y\in F} u(y)+ h_t(y,x),\text{ for all $t\in t_0$ and all $x\in M$,}$$
where $F=\{y\in M\mid \lvert u(y)\rvert<+\infty\}$.
We then observe that, for $y\in F$, the function $(t,x)\mapsto \varphi_y(t,x)=u(y)+h_t(y,x)$ is a
viscosity solution of the evolutionary Hamilton-Jacobi equation \eqref{HJE0} on $]0,+\infty[\times M$. 
Therefore, since $\hat u$ is continuous on $]0,t_0[\times M$ and is the $\inf$, on that set, of the family 
$\varphi_y, y\in F$, the function $\hat u$ is also a viscosity solution of the evolutionary Hamilton-Jacobi equation \eqref{HJE0} on $]0,t_0[\times M$. 
\end{proof}
\begin{rem}\label{ExistenceCalCurve} Note that once we know that $\hat u$ is continuous on $]0,t_0[\times M$, the $\inf$ in \eqref{FEXPT_}
is attained. Therefore, for every $(t,x)\in ]0,t_0[\times M$, we can find a $\hat u$-calibrated curve $\gamma: [a,t]\to M$
with $a<t$ and $\gamma(t)=x$. Since $\gamma$ is an extremal we can extend it to an extremal $\gamma:[0,t]\to M$.
Using that there is a $\hat u$-calibrated curve ending at any point of $]0,t_0[\times M$, it can be shown that
$\gamma$ is $\hat u$-calibrated  on every interval $[\epsilon,t]$, with $\epsilon>0$. With some more work,
 using mainly Proposition \ref{DefFABTR--} in an appropriate way, it can be shown that $\gamma$
 is $\hat u$-calibrated  on $[\epsilon,t]$. Details can be found in \cite{FatNC}.
\end{rem}
\end{appendix}

\end{document}